\documentclass[a4paper,10pt]{article}
\usepackage[utf8]{inputenc}

\usepackage{comment,xcomment,amssymb,amsmath,color}
\usepackage{rotating,xypic,array}
\usepackage{latexsym,mathtools,amsthm}
\usepackage{subfig,rotating}
\usepackage{tikz,fancyvrb}
\usepackage{booktabs}  
\usetikzlibrary{arrows.meta,decorations.pathmorphing}
\usepackage{enumitem}
\usepackage{url} 
\usepackage{longtable}
\usepackage[section]{placeins}
\usepackage[longtable]{multirow}
\usepackage[colorlinks,citecolor=blue]{hyperref}

\usepackage{pdflscape}

\makeatletter
\def\namedlabel#1#2{\begingroup
   \def\@currentlabel{#2}%
   \label{#1}\endgroup
}
\makeatother
\usepackage{geometry}

\title{Pseudo-K\"ahler and pseudo-Sasaki structures on Einstein solvmanifolds}
\author{Diego Conti and Federico A. Rossi and Romeo Segnan Dalmasso}

\DeclareFontFamily{U}{mathx}{\hyphenchar\font45}
\DeclareFontShape{U}{mathx}{m}{n}{
      <5> <6> <7> <8> <9> <10>
      <10.95> <12> <14.4> <17.28> <20.74> <24.88>
      mathx10
      }{}
\DeclareSymbolFont{mathx}{U}{mathx}{m}{n}
\DeclareFontSubstitution{U}{mathx}{m}{n}
\DeclareMathAccent{\widecheck}{0}{mathx}{"71}

\newtheorem{theorem}{Theorem}[section]
\newtheorem{lemma}[theorem]{Lemma}

\newtheorem{corollary}[theorem]{Corollary}
\newtheorem{proposition}[theorem]{Proposition}
\theoremstyle{definition}
\newtheorem{definition}[theorem]{Definition}
\newtheorem{example}[theorem]{Example}

\theoremstyle{remark}
\newtheorem{remark}[theorem]{Remark}

\newcommand{\R}{\mathbb{R}}
\newcommand{\lie}[1]{\mathfrak{#1}}     
\newcommand{\g}{\lie{g}}

\newcommand{\Q}{\mathbb{Q}}

\newcommand{\C}{\mathbb{C}}

\newcommand{\hook}{\lrcorner\,}
\newcommand{\LieG}[1]{\mathrm{#1}}      

\newcommand{\SU}{\mathrm{SU}}

\newcommand{\SO}{\mathrm{SO}}

\newcommand{\so}{\mathfrak{so}}
\newcommand{\co}{\mathfrak{co}}
\newcommand{\cu}{\mathfrak{cu}}

\newcommand{\su}{\mathfrak{su}}

\newcommand{\GL}{\mathrm{GL}}

\newcommand{\id}{\operatorname{Id}}   

\newcommand{\gl}{\lie{gl}}
\newcommand{\Sl}{\lie{sl}}
\newcommand{\Span}[1]{\operatorname{Span}\left\{#1\right\}}
\newcommand{\tran}[1]{\hspace{.2mm}\prescript{t\hspace{-.5mm}}{}{#1}}

\DeclareMathOperator{\ric}{ric} 
\DeclareMathOperator{\Ric}{Ric} 

\DeclareMathOperator{\Aut}{Aut}

\DeclareMathOperator{\diag}{diag}
\DeclareMathOperator{\Der}{Der}
\DeclareMathOperator{\ad}{ad}
\newcommand{\adcheck}{\widecheck{\operatorname{ad}}}

\DeclareMathOperator{\Ad}{Ad}
\DeclareMathOperator{\Tr}{tr}

\newcommand{\st}{\;\mid\;}          

\newcolumntype{C}{>{$}c<{$}}
\newcolumntype{L}{>{$}l<{$}}
\newcolumntype{R}{>{$}r<{$}}

\begin{document}
\VerbatimFootnotes
\maketitle
\begin{abstract}
The aim of this paper is to construct left-invariant Einstein pseudo-Riemannian Sasaki metrics on solvable Lie groups.

We consider the class of $\lie z$-standard Sasaki solvable Lie algebras of dimension $2n+3$, which are in one-to-one correspondence with pseudo-K\"ahler nilpotent Lie algebras of dimension $2n$ endowed with a compatible derivation, in a suitable sense. We characterize the pseudo-K\"ahler structures and derivations giving rise to Sasaki-Einstein metrics.

We classify $\lie z$-standard Sasaki solvable Lie algebras of dimension $\leq 7$ and those whose pseudo-K\"ahler reduction is an abelian Lie algebra.

The Einstein metrics we obtain are standard, but not of pseudo-Iwasawa type.
\end{abstract}

\renewcommand{\thefootnote}{\fnsymbol{footnote}}
\footnotetext{\emph{MSC class 2020}:  53C25;  53C30, 53C50, 22E25}
\renewcommand{\thefootnote}{\arabic{footnote}}

\section*{Introduction}

An effective method to construct Einstein metrics is by considering invariant metrics on a solvmanifold obtained by extending a suitable metric on a nilpotent Lie group of codimension one. Indeed, in the Riemannian case, Einstein solvmanifolds are described by a standard solvable Lie algebra $\widetilde{\g}$ of Iwasawa type (\cite{Heber:noncompact,Lauret:Einstein_solvmanifolds}). In particular, this means that $\widetilde{\g}$ admits an orthogonal decomposition $\widetilde{\g}=\g\rtimes\lie a$, with $\g$  nilpotent, $\lie a$ abelian and $\ad X$ self-adjoint whenever $X$ is in $\lie a$. Furthermore, the restriction of the metric to $\g$ satisfies the so-called nilsoliton equation (\cite{Lauret:RicciSoliton}).

Things are more complicated in the indefinite case (see e.g.~\cite{ContiRossi:RicciFlat}), but it is still possible to construct  Einstein solvmanifolds by extending a nilsoliton; indeed, there is a correspondence between nilsolitons and a class of Einstein solvmanifolds for which $\widetilde{\g}$ admits a decomposition as above, called a pseudo-Iwasawa decomposition (see~\cite{ContiRossi:IndefiniteNilsolitons}).

In the non-invariant setting, Einstein metrics are often studied in the presence of additional structures, such as a Killing spinor or a restriction on the holonomy (see e.g. \cite{Bar1993RealHolonomy,Baum2014OnManifolds}). It is then natural to ask whether Einstein metrics compatible with such special structures can be obtained in the invariant setting too.

This paper is focused on Sasaki metrics. More precisely, we consider a class of left-invariant pseudo-Riemannian Sasaki-Einstein metrics on solvable Lie groups. Sasaki-Einstein metrics admit a Killing spinor (see~\cite{FriedrichKath:SevenDimensionalCompact}), and may be viewed as the odd-dimensional counterpart of K\"ahler-Einstein geometry. We showed in~\cite{ContiRossiSegnan:PseudoSasaki} that Sasaki Lie algebras can never be of pseudo-Iwasawa type, regardless of whether they are Einstein; therefore, Sasaki-Einstein solvmanifolds cannot be obtained by extending a nilsoliton.

We therefore consider a different construction, represented pictorially in Figure~\ref{fig:diagrammone}; we illustrate it in the invariant case, though it holds more generally.

Suppose that $\widetilde{G}$ is a solvable Lie group with a left-invariant pseudo-Riemannian Sasaki structure and $X$ is a left-invariant, nonnull Killing field preserving the structure; one can then consider the moment map $\widetilde{\mu}$ and the contact reduction $\widetilde{G}//X$. Suppose that the zero-level set $G$ of the moment map is a nilpotent Lie subgroup, and $X$ is in the center of its Lie algebra; then the contact reduction $\widetilde{G}//X$ is also a Lie group with a left-invariant Sasaki structure. Notice that we do not assume that $X$ is in the center of the Lie algebra of $\widetilde{G}$; therefore, the quotient $\widetilde{G}/\mathcal{F}_X$, which we omitted in the diagram, is not necessarily a Lie group.

Assuming that the Reeb vector field $\xi$ is central, we have that the quotient by the Reeb foliation $\widetilde{G}/\mathcal{F}_\xi$ is also a Lie group, with a left-invariant pseudo-K\"ahler structure. Then its symplectic reduction is a pseudo-K\"ahler nilpotent Lie group $\widecheck{G}$, which can also be described as the quotient of the contact reduction $\widetilde{G}//X$ by the Reeb foliation. In accord with~\cite{ContiRossiSegnan:PseudoSasaki}, we call $\widecheck{G}$ the \emph{K\"ahler reduction} of $\widetilde{G}$. Notice that $\widecheck{G}$ is Ricci-flat; this is a general property of pseudo-K\"ahler nilpotent Lie groups (see~\cite{FinoPartonSalamon}).

\begin{figure}[t]
{\centering
\begin{tikzpicture}[scale=1]
\begin{scope}[every node/.style={circle,thick}, minimum size=3em]
        \node (a1) at (0,4) {$\widetilde{G}$};
        \node (a2) at (-4,2.5) {$\widetilde{G}/\mathcal{F}_\xi$};
        \node (a3) at (4,2.5) {$G$};
        \node (a4) at (-4,0) {$\widetilde{G}/\mathcal{F}_{\xi,X}$};
        \node (a5) at (0,1) {$G/\mathcal{F}_\xi$};
        \node (a6) at (4,0) {$\widetilde{G}//X$};
        \node (a7) at (0,-1.5) {$\widecheck{G}$};
\end{scope}
\begin{scope}[>={stealth[black]}, every edge/.style={draw, thin}, every edge quotes/.style = { align=center, inner sep=1pt}]
    \path [->] (a1) edge node [above,midway]{\footnotesize{$/\mathcal{F}_\xi$}} (a2);
    \path [->] (a3) edge node [below,pos=0.7]{\footnotesize{$/\mathcal{F}_\xi$}} (a5);
    \path [->] (a6) edge node [below,midway]{\footnotesize{$/\mathcal{F}_\xi$}} (a7);
    \path [->] (a2) edge node [left,midway]{\footnotesize{$/\mathcal{F}_X$}} (a4);
    \path [->] (a3) edge node [right,midway]{\footnotesize{$/\mathcal{F}_X$}} (a6);
    \path [->] (a5) edge node [right,midway]{\footnotesize{$/\mathcal{F}_X$}} (a7);
    \draw [->, decorate, thin, decoration={snake,amplitude=.4mm,segment length=2mm, post length=1mm}] (a1) -- (a3)  node [above right,midway,align=center]{\footnotesize{$\widetilde{\mu}^{-1}(0)$}};
    \draw [->, decorate, thin, decoration={snake,amplitude=.4mm,segment length=2mm, post length=1mm}] (a1) -- (a6)  node [below left,pos=0.3,align=center]{\footnotesize{$//X$}};
    \draw [->, decorate, thin, decoration={snake,amplitude=.4mm,segment length=2mm, post length=1mm}] (a2) -- (a5)  node [above right,midway,align=center]{\footnotesize{$\widetilde{\mu}^{-1}(0)$}};
    \draw [->, decorate, thin, decoration={snake,amplitude=.4mm,segment length=2mm, post length=1mm}] (a2) -- (a7)  node [above right,pos=0.6,align=center]{\footnotesize{$//X$}};
    \draw [->, decorate, thin, decoration={snake,amplitude=.4mm,segment length=2mm, post length=1mm}] (a4) -- (a7)  node [below left,midway,align=center]{\footnotesize{$\widetilde{\mu}^{-1}(0)$}};
    \end{scope}
\end{tikzpicture}
\caption{\label{fig:diagrammone}
Contact and symplectic reduction of a Sasaki Lie group $\widetilde{G}$.\\
Dimension increases going up, from $\dim\widecheck{G}=2n$ at the bottom to $\dim\widetilde{G}=2n+3$ at the top. The arrows represent a \emph{determines}-type relation; straight arrows are used for quotients (bundles), and curly arrows for other constructions, namely symplectic/contact quotient and extraction of a level set of the moment maps $\mu$ and $\widetilde{\mu}$.
}
}
\end{figure}

Our aim is to obtain Sasaki-Einstein solvmanifolds by inverting the diagram of Figure~\ref{fig:diagrammone}. The straight arrows have natural inverses: one takes a circle bundle with curvature determined by the K\"ahler form or $dX^\flat$ (see~\cite{Hatakeyama}). By contrast, the curly arrows are not bijections in general: if $\nabla X^\flat=dX^\flat$ is known, only the metric on $\mu^{-1}(0)$ and the second fundamental form are determined, but this does not determine the metric on the full group $\widetilde{G}$, for general Sasaki metrics. However, the metric is determined by its restriction to the hypersurface and its second fundamental form if one requires $\widetilde{G}$ to be Einstein, according to~\cite{Koiso:hypersurfaces}.

More explicitly, symplectic reduction can be inverted as follows. The pseudo-K\"ahler Lie group $\widecheck{G}$ is endowed with a closed $(1,1)$-form $\gamma$, corresponding to $dX^\flat$. We can describe $\widetilde{G}/\mathcal{F}_\xi$ as the product of a circle bundle with curvature $\gamma=d\theta$ and a line, endowed with the complex structure for which $(1,0)$-forms are generated by the pull-back of $(1,0)$-forms on $\widecheck{G}$ along with $dr+ie^{2r}\theta$ and the K\"ahler form
\begin{equation}
 \label{eqn:flippedberardbergery}
e^{2r}\omega -  dr\wedge (e^{2r} \theta),
\end{equation}
where $\omega$ denotes the K\"ahler form of $\widecheck{G}$. It turns out that this can be written as a left-invariant metric on a solvable Lie group $\widetilde{G}/\mathcal{F}_\xi$.

In our invariant setup, the hypersurface $G$ in $\widetilde{G}$ corresponds to a nilpotent ideal of codimension one in a solvable Lie algebra, and the construction can be
studied with the language of standard Lie algebras. However, rather than the Lie algebras of (pseudo-)Iwasawa type studied in \cite{Heber:noncompact,ContiRossi:IndefiniteNilsolitons}, we need  to consider a different class, that we introduced in~\cite{ContiRossiSegnan:PseudoSasaki} under the name of  $\lie z$-standard Sasaki Lie algebras. This condition means that the Lie algebra $\widetilde{\g}$ is endowed with a Sasaki structure $(\phi,\xi,\eta,g)$ such that $\widetilde{\g}$ takes the form of an orthogonal semidirect product $\g\rtimes\Span{e_0}$, with $\phi(e_0)$  a  central element of $\g$ corresponding to the vector field $X$ in Figure~\ref{fig:diagrammone} (see Section~\ref{sec:preliminaries} for the precise definitions). It then becomes possible to characterize and study the construction of Figure~\ref{fig:diagrammone} in purely algebraic terms. In particular, we showed in~\cite{ContiRossiSegnan:PseudoSasaki} that the Lie algebra $\widecheck{\g}$ of the K\"ahler reduction  comes endowed with a derivation $\widecheck{D}$ satisfying certain conditions; in Figure~\ref{fig:diagrammone}, $\widecheck{D}$ determines $\widetilde{G}/\mathcal{F}_{\xi,X}$ as a semidirect product $\widecheck{G}\rtimes\R$.

\smallskip
In this paper, we specialize this construction to the Sasaki-Einstein case. As a first step, we introduce a generalization of the nilsoliton condition that enables one to construct Einstein solvmanifolds which are not of pseudo-Iwasawa type, by taking a semidirect product with $\R$; these metrics are not necessarily Sasaki (Proposition~\ref{prop:generaleinstein}). We then characterize $\lie z$-standard Sasaki-Einstein solvable Lie algebras in terms of their K\"ahler reduction $\widecheck{\g}$, showing that the symmetric part of the derivation $\widecheck{D}$ is the identity and preserves the pseudo-K\"ahler structure. This implies that $\widecheck{D}$ lies in the Lie algebra $\Der\widecheck{\g}\cap\cu(p,q)$, where $\cu(p,q)=\lie u(p,q)\oplus\Span \id$.

In the opposite direction, as sketched in Figure~\ref{figure:diagrammonealgebre}, we show that any pseudo-K\"ahler nilpotent Lie algebra with a derivation $\widecheck{D}$ whose symmetric part is the identity induces:
\begin{itemize}
 \item an Einstein metric on $\widecheck{\g}\rtimes_{\widecheck{D}}\R$, corresponding to $\widetilde{G}/\mathcal{F}_{\xi,X}$  in Figure~\ref{fig:diagrammone};
 \item a pseudo-K\"ahler-Einstein metric with positive curvature on a solvable double extension $\widetilde{\lie{k}}$, corresponding to $\widetilde{G}/\mathcal{F}_{\xi}$. Geometrically, this corresponds to setting $\gamma=2\omega$ and applying the ansatz~\eqref{eqn:flippedberardbergery}.  Notice that if one flips a sign in~\eqref{eqn:flippedberardbergery} and writes $e^{2r}\omega + e^{2r} dr\wedge \theta$, one obtains the construction of a (Riemannian) K\"ahler-Einstein metric with negative curvature on the bundle  over a K\"ahler Ricci-flat manifold given in~\cite[\S~11.8]{BerardBergery:Sur} (see also \cite[Equation~(3.20)]{PagePope:Inhomogeneous} and \cite[Theorem~9.129]{Besse}).
 \item a Sasaki-Einstein metric on a central extension of $\widetilde{\lie{k}}$, corresponding to $\widetilde{G}$ in the diagram. Notice that having chosen the metric with positive curvature on  $\widetilde{\lie{k}}$ is essential for this step, due to the constraints on the space of leaves of a Sasaki-Einstein manifold.
 \end{itemize}
We then show that any two choices of $\widecheck{D}$ on $\widecheck{\g}$ with symmetric part equal to the identity determine isometric Einstein extensions; the isometry is at the level of solvmanifolds as pseudo-Riemannian manifolds, and it does not preserve the Lie algebra structure (Theorem~\ref{thm:metricnikSEextension}).

It turns out that $\Der\widecheck{\g}\cap\cu(p,q)$ contains elements $\widecheck{D}$ with symmetric part equal to the identity if and only if it contains an element with nonzero trace; in that case, we show that the Lie algebra contains a canonical choice of $\widecheck{D}$. This canonical element is obtained by adapting a construction of Nikolayevsky. Indeed, we fix an algebraic subalgebra $\lie h$ of $\gl(m,\R)$, and define the $\lie h$-Nikolayevsky derivation on a Lie algebra $\g$ with a $H$-structure as the unique semisimple derivation $N$ in $\lie h\cap\Der\g$ such that
\[\Tr(\psi N)=\Tr \psi, \qquad \psi\in\lie h\cap\Der\g.\]
For $\lie h=\gl(m,\R)$, one obtains the Nikolayevsky derivation introduced in~\cite{Nikolayevsky}, and for $\lie h=\co(p,q)$ the metric Nikolayevsky derivation of~\cite{ContidelBarcoRossi:Uniqueness}. Existence and uniqueness of the $\lie h$-Nikolayevsky derivation is proved similarly as in these particular cases (Proposition~\ref{prop:generalizednik}).

The relevant situation for this paper is the $\cu(p,q)$-Nikolayevsky derivation of a pseudo-K\"ahler Lie algebra, which turns out to have rational eigenvalues, like the ordinary Nikolayevsky derivation. Thus, we see that the element of $\Der\widecheck{\g}\cap\cu(p,q)$ that determines the Sasaki-Einstein extension can be assumed to be diagonalizable over $\Q$ (Proposition~\ref{prop:rationaleigenvalues}). The existence of a nonzero $\cu(p,q)$-Nikolayevsky derivation guarantees that there is a standard Einstein extension, corresponding to $\widetilde{G}/\mathcal{F}_{\xi,X}$ in Figure~\ref{fig:diagrammone}.

The characterization of $\lie z$-standard Sasaki-Einstein solvmanifolds in terms of their K\"ahler reduction allows us to classify all $\lie z$-standard Sasaki-Einstein solvmanifolds of dimension $7$ (Theorem~\ref{thm:classification}). In addition, we are able to write down all $\lie z$-standard Sasaki-Einstein solvmanifolds that reduce to a pseudo-K\"ahler $\widecheck{\g}$ which is abelian as a Lie algebra (Corollary~\ref{cor:classify_abelian}). This includes all Lorentzian $\lie z$-standard Sasaki-Einstein solvmanifolds, for which $\widecheck{\g}$ is forced to be a nilpotent K\"ahler Lie algebra, hence abelian.

In particular, our results give rise to explicit pseudo-K\"ahler-Einstein and Sasaki-Einstein solvmanifolds in all dimensions $\geq4$ (Theorem~\ref{thm:classification}, Remark~\ref{rem:keexplicit}). These metrics are not Ricci-flat, which is a general fact for Sasaki-Einstein metrics and their K\"ahler-Einstein quotients.

We point out that our Sasaki-Einstein metrics are examples of Einstein standard solvmanifolds that are not isometric to any Einstein solvmanifold of pseudo-Iwasawa type. This is in sharp contrast to the Riemannian case, where~\cite{Heber:noncompact} shows that all standard Einstein solvmanifolds are of Iwasawa type up to isometry.

We also show with an example that not all pseudo-K\"ahler Lie algebras can be extended to a $\lie z$-standard Sasaki-Einstein Lie algebra (Example~\ref{ex:doesnotextend}).

\smallskip
\textbf{Acknowledgments}
This paper was written as part of the PhD thesis of the third author, written under the supervision of the first author, for the joint PhD programme in Mathematics Università di Milano Bicocca -- University of Surrey.

The authors thank the anonymous referee for detailed and constructive feedback.

The authors acknowledge GNSAGA of INdAM.

\section{Preliminaries: structures on Lie algebras}
\label{sec:preliminaries}
In this section we introduce some general language relevant to the study  of Sasaki-Einstein metrics, specialized to the invariant setting, and recall some results that will be needed in the sequel.

Given a Lie algebra $\g$ of dimension $m$, we can think of a basis of $\g$ as a frame $\R^m\cong\g$. There is a natural right action of $\GL(m,\R)$ on the set of frames. Given  a subgroup $H\subset\GL(m,\R)$, we will say that a \emph{$H$-structure on $\g$} is a $H$-orbit in the space of frames. Given any frame $u$, the identification $u\colon\R^m\cong\g$ induces a left action of $H$ on $\g$. This induces an inclusion map $H\to\GL(\g)$ that depends on the frame $u$, but the \emph{image} of the inclusion only depends on the $H$-structure. Accordingly, whenever we have a $H$-structure on $\g$, we will write $H\subset\GL(\g)$, $\lie h\subset\gl(\g)$.

It is clear that a $H$-structure on a Lie algebra $\g$ induces a left-invariant $H$-structure, in the usual sense, on any Lie group with Lie algebra $\g$.

An \emph{almost contact structure} on a $(2n + 1)$-dimensional Lie algebra $\widetilde{\g}$ is a triple $(\phi, \xi, \eta)$, where $\phi$ is a linear map from $\widetilde{\g}$ to itself, $\xi$ is an element of $\widetilde{\g}$, $\eta$ is in $\widetilde{\g}^*$ and
\[\eta(\xi)=1,\qquad \eta\circ \phi=0,\qquad \phi^2 =-\id+\eta\otimes\xi.\]
Given a nondegenerate scalar product $g$ on $\widetilde{\g}$, the quadruple $(\phi ,\xi,\eta,g)$ is called \emph{an almost contact metric structure} if $(\phi, \xi, \eta)$ is an almost contact structure and
\[g(\xi,\xi)=1,\qquad \eta= \xi^{\flat},\qquad g(\phi X,\phi Y)=g(X,Y)- \eta(X)\eta(Y),\]
for any $X,Y\in \g$. One then defines a two-form $\Phi$ by $\Phi(X,Y)=g(X,\phi Y)$.

Given an almost contact metric structure of signature $(2p+1,2q)$ on $\g$, one can find a frame $e_1,\dotsc, e_{2n+1}$ with dual basis $\{e^i\}$ such that
\begin{gather*}
g=e^1\otimes e^1+\dotsc + e^{2p}\otimes e^{2p}-e^{2p+1}\otimes e^{2p+1}-\dots - e^{2p+2q}\otimes e^{2p+2q}+e^{2p+2q+1}\otimes e^{2p+2q+1}, \\
\eta=e^{2p+2q+1},\qquad \Phi=e^{12}+\dotsc + e^{2p-1,2p}-e^{2p+1,2p+2}-\dotsc - e^{2p+2q-1,2p+2q}.
\end{gather*}
The common stabilizer of these tensors in $\GL(2p+2q+1,\R)$ is $\LieG{U}(p,q)$. Thus, an almost contact metric structure on a Lie algebra can be viewed as a $\LieG{U}(p,q)$-structure.

An almost contact metric structure is called \emph{Sasaki} if
$N_{\phi} + d\eta \otimes \xi = 0$ and $d\eta = 2\Phi$, where $N_\phi$ denotes the Nijenhuis tensor.

These definitions mimic analogous definitions for structures on manifolds (see e.g. \cite{BoGa:SasakianBook,CalvCaLo:PseudoRiemHomBook}). It is clear that a Sasaki structure on $\widetilde{\g}$ defines a left-invariant almost Sasaki structure on any Lie group $\widetilde{G}$ with Lie algebra $\widetilde{\g}$ by left translation. In particular, $g$ defines a pseudo-Riemannian metric on $\widetilde{G}$. We shall also refer to $g$ as a \emph{metric} on $\widetilde{\g}$, and define its Levi-Civita connection, curvature and so on in terms of the corresponding objects on $\widetilde{G}$.

Sasaki structures are an odd-dimensional analogue of (pseudo)-K\"ahler structures. In our invariant setting, a pseudo-K\"ahler structure on a Lie algebra $\g$ is a triple $(g,J,\omega)$, where $g$ is a pseudo-Riemannian metric, $J\colon\g\to\g$ is an almost complex structure satisfying the compatibility condition $g(JX,JY)=g(X,Y)$, and $\omega(X,Y)=g(X,JY)$, where one further imposes that $N_J=0$ and $d\omega=0$.

Having fixed the metric, for any endomorphism $f\colon\widetilde{\g}\to\widetilde{\g}$ we write $f=f^s+f^a$, where $f^s$ is symmetric and $f^a$ is skew-symmetric relative to the metric, i.e.
\[f^s=\frac12(f+f^*), \qquad f^a=\frac12(f-f^*).\]
With this notation, the Levi-Civita connection is given by
\begin{equation}
\label{eqn:lc}
\nabla_w v=-\ad(v)^s w-\frac12 (\ad w)^* v.
\end{equation}
The Ricci tensor can be written as
\begin{equation}
\label{eqn:ricciwarlike}
2\ric(v,w)=-\Tr \ad (v\hook dw^\flat+w\hook dv^\flat)+g(dv^\flat,dw^\flat)-g(\ad v, \ad w)-\Tr(\ad v\circ \ad w),
\end{equation}
see e.g. \cite[Lemma~1.1]{ContiRossi:EinsteinNilpotent}.
If $\widetilde{\g}$ is unimodular with Killing form equal to zero, this formula simplifies to
\begin{equation}
\label{eqn:riccipeaceful}
2\ric(v,w)=g(dv^\flat,dw^\flat)-g(\ad v, \ad w).
\end{equation}

We will need a result originaly proved in~\cite{AzencottWilson2} for Riemannian metrics and later adapted to standard indefinite metrics in~\cite{ContiRossi:IndefiniteNilsolitons}, though the standard condition turned out not to be necessary (see~\cite{ContiRossiSegnan:PseudoSasaki}). The precise statement we are going to need is the following:
\begin{proposition}[\cite{AzencottWilson2,ContiRossiSegnan:PseudoSasaki}]
\label{prop:pseudoAzencottWilson}
Let $H$ be a subgroup of $\SO(r,s)$ with Lie algebra $\lie h$ and $\widetilde{\g}$ a Lie algebra of the form $\widetilde{\g}=\g\rtimes \lie{a}$ endowed with a $H$-structure. Let $\chi\colon\lie a\to\Der(\g)$ be a Lie algebra homomorphism such that, extending $\chi(X)$ to $\widetilde{\g}$ by declaring it to be zero on $\lie a$,
\begin{equation}
\label{eqn:adXstar}
\chi(X)-\ad X\in\lie h, \qquad [\chi (X),\ad Y]=0,\ X,Y\in\lie a.
\end{equation}
Let $\widetilde{\g}^*$ be the Lie algebra $\g\rtimes_\chi\lie{a}$. If $\widetilde{G}$ and $\widetilde{G}^*$ denote the connected, simply connected Lie groups with Lie algebras $\widetilde{\g}$ and $\widetilde{\g}^*$, with the corresponding left-invariant $H$-structures, there is an isometry from $\widetilde{G}$ to $\widetilde{G}^*$, whose differential at $e$ is the identity of $\g\oplus\lie{a}$ as a vector space, mapping the $H$-structure on $\widetilde{G}$ into the $H$-structure on $\widetilde{G}^*$.
\end{proposition}
\begin{proof}
For $\lie h=\so(r,s)$, $\chi(X)-\ad X$ is skew-symmetric and the proof is identical to {\cite[Proposition~2.2]{ContiRossiSegnan:PseudoSasaki}}. In general, one uses that $\chi(X)-\ad X$ is in $\lie h$ to conclude that the action of $G^*$ on $\widetilde{G}$ preserves the $H$-structure.
\end{proof}
We will say that two Lie algebras endowed with a $H$-structure are \emph{equivalent} if there is an isometry between the corresponding simply-connected Lie groups mapping one $H$-structure into the other.

Recall from~\cite{ContiRossi:IndefiniteNilsolitons} that a \emph{standard decomposition} of the Lie algebra with a fixed metric is an orthogonal splitting
\[\widetilde{\g}=\g\rtimes^\perp\lie a,\]
where $\g$ is a nilpotent ideal and $\lie a$ is an abelian subalgebra. This definition generalizes the definition given in~\cite{Heber:noncompact} for positive-definite metrics.

In~\cite{ContiRossiSegnan:PseudoSasaki}, we introduced a special class of standard Sasaki Lie algebras: if $\widetilde{\g}$ has both a Sasaki structure  $(\phi, \xi, \eta,g)$ and a standard decomposition of the form $\widetilde{\g}=\g\rtimes\Span{e_0}$, it is called \emph{$\lie z$-standard} if $\phi(e_0)$ lies in the center $\lie z$ of $\g$. This means that the one-parameter group $\{\exp tb\}$, with $b=\phi(e_0)$, acts on the corresponding group in such a way that the contact quotient is still a Lie group; this  implies that the pseudo-K\"ahler Lie group obtained by quotienting by the Reeb direction has a symplectic reduction which is a pseudo-K\"ahler Lie group, motivating the following:
\begin{definition}
Given  a $\lie z$-standard Lie algebra $\widetilde{\g}=\g\rtimes\Span{e_0}$ with Sasaki structure $(\phi, \xi, \eta,g)$, the quotient Lie algebra $\widecheck{\g}=\g/\Span{b,\xi}$ with the metric induced by $g$ and complex structure induced by $\phi$ is called
the \emph{K\"ahler reduction} of $\widetilde{\g}$.
\end{definition}
In fact, $\lie z$-standard Sasaki Lie algebras can be characterized in terms of their K\"ahler reduction as follows (Corollary~4.4 and Proposition~5.1 in~{\cite{ContiRossiSegnan:PseudoSasaki}}):
\begin{proposition}[{\cite[Proposition~5.1 and Corollary~4.4]{ContiRossiSegnan:PseudoSasaki}}]
\label{prop:constructive}
Let $(\widecheck{\g},J,\omega)$ be a pseudo-K\"ahler nilpotent Lie algebra. Let $\widecheck{D}$ be a derivation of $\widecheck{\g}$, $\tau=\pm1$, and $\g=\widecheck{\g}\oplus\Span{b,\xi}$ a central extension of $\g$ with a metric of the form:
\[g(x,y)=\widecheck{g}(x,y), \qquad g(x,b)=0=g(x,\xi), \qquad g(\xi,\xi)=1, \qquad g(b,b)=\tau,\qquad g(b,\xi)=0,\]
where $x,y\in\widecheck{\g}$.
Assume furthermore
\begin{itemize}
\item $d\xi^\flat=2\omega$, where the right-hand-side is implicitly pulled back to $\widetilde{\g}$;
\item $db^\flat=\widecheck{D}\omega$, where the right-hand-side is implicitly pulled back to $\widetilde{\g}$;
\item $[J,\widecheck{D}]=0$;
\item $[\widecheck{D}^s,\widecheck{D}^a]=h\widecheck{D}^s-2(\widecheck{D}^s)^2$ for some constant $h$.
\end{itemize}
Let $\widetilde{\g}=\g\rtimes \Span{e_0}$, where
\[[e_0,x]=\widecheck{D}x, \qquad [e_0,b]=hb-2\tau \xi, \qquad [e_0,\xi]=0;\]
then $\widetilde{\g}$ has a $\lie z$-standard Sasaki structure $(\phi,\eta,\xi,\widetilde{g})$ given by
\[\widetilde{g}=g+\tau e^0\otimes e^0, \qquad \phi(x)=J(x)+\tau g(b,x)e_0, \qquad \phi(e_0)=-b, \quad x\in\g.\]
Conversely, every  $\lie z$-standard Sasaki Lie algebra arises in this way.
\end{proposition}
Figure~\ref{figure:diagrammonealgebre}, which should be compared with Figure~\ref{fig:diagrammone}, summarizes the Lie algebras appearing in Proposition~\ref{prop:constructive} and their relations, alongside other related Lie algebras which will appear in the sequel of the paper, namely:
\begin{itemize}
\item[$\widecheck{\g}$] a pseudo-K\"ahler nilpotent Lie algebra of dimension $2n$;
\item[$\g_{\mathrm{std}}$] a solvable standard extension of $\widecheck{\g}$ by the derivation $\widecheck{D}$, also a quotient of $\widetilde{\lie k}$ by the non-central one-dimensional ideal $\Span{b}$;
\item[$\lie k$] a nilpotent central extension of $\widecheck{\g}$ by the cocycle $db^\flat=\widecheck{D}\omega$;
\item[$\g^\circlearrowright$] a nilpotent Sasaki central extension of $\widecheck{\g}$ by the cocycle $d\xi^\flat=2\omega$;
\item[$\widetilde{\lie k}$] a solvable K\"ahler Lie algebra, which can be obtained as a standard extension of $\lie k$ by the derivation $\widecheck{D}+2\tau b^\flat\otimes b$;
\item[$\g$] a nilpotent central extension of $\widecheck{\g}$ by $db^\flat=\widecheck{D}\omega$ and $d\xi^\flat=2\omega$;
\item[$\widetilde{\g}$] a $\lie z$-standard Sasaki Lie algebra of dimension $2n+3$.
\end{itemize}
\begin{remark}
Composing a semidirect product with a central extension in such a way that the two new directions span an indefinite two-plane is a procedure known as double extension (see~\cite{MedinaRevoy}). Our construction is different because the semidirect products and central extensions corresponding to vertical arrows in Figure~\ref{figure:diagrammonealgebre} have the effect of adding a definite two-plane.
\end{remark}

\begin{figure}[t]
{\centering
\begin{tikzpicture}[scale=1]
\begin{scope}[every node/.style={circle,thick}, minimum size=3em]
        \node (a1) at (0,4) {$\widetilde{\g}$};
        \node (a2) at (-4,2.5) {$\widetilde{\lie{k}}$};
        \node (a3) at (4,2.5) {$\g$};
        \node (a4) at (-4,0) {$\g_{\mathrm{std}}$};
        \node (a5) at (0,1) {$\lie{k}$};
        \node (a6) at (4,0) {$\g^\circlearrowright$};
        \node (a7) at (0,-1.5) {$\widecheck{\g}$};
\end{scope}
\begin{scope}[>={stealth[black]}, every edge/.style={draw, thin}, every edge quotes/.style = { align=center, inner sep=1pt}]
    \path [->] (a3) edge node [above right,midway]{\footnotesize{$\rtimes e_0$}} (a1);
    \path [->] (a5) edge node [above right,midway]{\footnotesize{$\rtimes e_0$}} (a2);
    \path [->] (a7) edge node [below left,midway]{\footnotesize{$\rtimes e_0$}} (a4);
    \draw [->, decorate, thin, decoration={snake,amplitude=.4mm,segment length=2mm, post length=1mm}] (a2) -- (a1)  node [above left,midway,align=center]{\footnotesize{$\lie{z}$-ext}};
    \draw [->, decorate, thin, decoration={snake,amplitude=.4mm,segment length=2mm, post length=1mm}] (a7) -- (a6)  node [below right,midway,align=center]{\footnotesize{$\lie{z}$-ext}};
    \draw [->, decorate, thin, decoration={snake,amplitude=.4mm,segment length=2mm, post length=1mm}] (a6) -- (a3)  node [right,midway,align=center]{\footnotesize{$\lie{z}$-ext}};
    \draw [->, decorate, thin, decoration={snake,amplitude=.4mm,segment length=2mm, post length=1mm}] (a7) -- (a5)  node [right,midway,align=center]{\footnotesize{$\lie{z}$-ext}};
    \draw [dotted, thick] (a4) -- (a2)  node {};
    \draw [->, decorate, thin, decoration={snake,amplitude=.4mm,segment length=2mm, post length=1mm}] (a7) -- (a2)  node [fill=white,font=\footnotesize, midway]{{\footnotesize{Cor.~\ref{cor:constructke}}}};
    \draw [->, decorate, thin, decoration={snake,amplitude=.4mm,segment length=2mm, post length=1mm}] (a7) to[bend right=2.5cm, in=225, out=315] node [fill=white,font=\footnotesize, midway]{\footnotesize{Prop.~\ref{prop:constructse}}}  (a1);
\end{scope}
\end{tikzpicture}
\caption{\label{figure:diagrammonealgebre}Construction of a $\lie z$-standard Sasaki Lie algebra from its K\"ahler reduction. Dimension increases going up, from $\dim\widecheck{\g}=2n$ at the bottom to $\dim\widetilde{\g}=2n+3$ at the top.}
}
\end{figure}
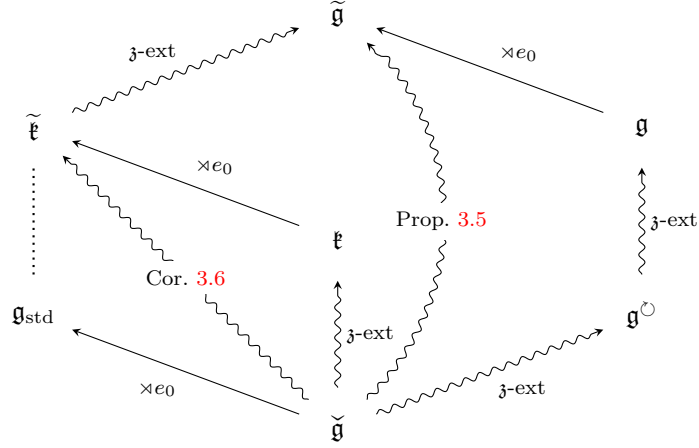

It is well known that, given a Sasaki manifold $M$, the space of leaves of the Reeb foliation has a pseudo-K\"ahler structure (see~\cite{Ogiue:OnFiberings}), and the Ricci tensor of the latter is determined by the Ricci tensor of $M$ (see \cite[Theorem~7.3.12]{BoGa:SasakianBook}). In our invariant setting, this fact takes the following form:
\begin{proposition}
\label{prop:kahlerquotient}
Let $\g$ be a Lie algebra with a Sasaki structure $(\phi,\xi,\eta, g)$. Suppose $\g$ has nonzero center. Then $\lie z(\g)=\Span{\xi}$ and the quotient $\widecheck{\g}=\g/\Span{\xi}$ has an induced pseudo-K\"ahler structure $(\widecheck{g}, J,\omega)$ with
\[\widecheck{\ric}=\ric + 2\widecheck{g},\]
where $\ric$ is restricted to $\xi^\perp$ implicitly.
\end{proposition}
\begin{proof}
Any element of the center satisfies $v\hook d\eta=0$, so it is a multiple of $\xi$. Thus, the kernel coincides with $\Span{\xi}$.

As a vector space, we identify $\widecheck{\g}$ with $\xi^\perp$, so that the metric $\widecheck{g}$ is the restriction of $g$. The Lie algebra structure of $\widecheck{\g}$ is given by a projection, i.e.
\[\ad(v)(w)=\adcheck(v)(w)-d\eta(v,w)\xi.\]
Therefore
\[\ad(v)^*(w)=\adcheck(v)^*(w), \qquad \ad(v)^*(\xi)=-(v\hook d\eta)^\sharp.\]
Then, from equation~\eqref{eqn:lc}, the Levi-Civita connections $\nabla$, $\widecheck{\nabla}$ are related by
\begin{gather*}
\begin{split}
\nabla_w v&=-\ad(v)^s(w)-\frac12\ad(w)^*(v)\\
&=-\adcheck(v)^s(w)+\frac12d\eta(v,w)\xi-\frac12\adcheck(w)^*(v)=\widecheck{\nabla}_w v+\frac12d\eta(v,w)\xi;
\end{split}\\
\nabla_w \xi=-\ad(w)^s(\xi)-\frac12\ad(\xi)^*(w)=\frac12(w\hook d\eta)^\sharp.
\end{gather*}
If $\alpha\in\operatorname{Ann}(\xi)$, we have
\begin{gather*}
(\nabla_w \alpha)(v)=-\alpha(\nabla_wv)=-\alpha(\widecheck{\nabla}_wv)=(\widecheck{\nabla}_w\alpha)(v),\\
(\nabla_w \alpha)(\xi)=-\alpha(\nabla_w\xi)= -\frac12g(w\hook d\eta,\alpha).
\end{gather*}
The Sasaki condition implies
\[\nabla_v d\eta = 2\eta\wedge v^\flat,\]
so $\widecheck{\nabla}_v d\eta=0$. This implies that $d\eta$ defines a pseudo-K\"ahler structure on $\widecheck{\g}$.

The exterior derivative $\widecheck{d}$ on $\widecheck{\g}$ can be identified with the restriction of $d$. By~\eqref{eqn:ricciwarlike}, we obtain
\begin{multline*}
2\ric(v,w)=
-\Tr \ad (v\hook dw^\flat+w\hook dv^\flat)+g(dv^\flat,dw^\flat)-g(\ad v, \ad w)-\Tr(\ad v\circ \ad w)\\
=-\Tr \adcheck (v\hook \widecheck{d}w^\flat+w\hook \widecheck{d}v^\flat)+g(\widecheck{d}v^\flat, \widecheck{d}w^\flat)-g(\adcheck v,\adcheck w)-g(v\hook d\eta,w\hook d\eta)-\Tr(\adcheck v\circ \adcheck w)\\
=2\widecheck{\ric}(v,w)-g(2\phi(v),2\phi(w))=2\widecheck{\ric}(v,w)-4g(v,w).\qedhere
\end{multline*}
\end{proof}
\begin{remark}
Every Sasaki metric on a manifold of dimension $2n+1$ satisfies $\Ric(\xi)=2n\xi$. Accordingly, the space of leaves of a Sasaki-Einstein manifolds satisfies
\begin{equation}
 \label{eqn:ricciofkahlerquotient}
 \widecheck{\ric}=(2n+2)\widecheck{g}.
\end{equation}
Conversely, every K\"ahler-Einstein manifold with positive scalar curvature, suitably normalized so as to satisfy~\eqref{eqn:ricciofkahlerquotient}, gives rise to a Sasaki-Einstein manifold in one dimension higher.
\end{remark}

We recall the following fact from~\cite{FinoPartonSalamon}:
\begin{lemma}[{\cite[Lemma~6.3]{FinoPartonSalamon}}]
\label{lemma:finopartonsalamon}
Pseudo-K\"ahler metrics on nilpotent Lie algebras are Ricci-flat.
\end{lemma}
\begin{proof}
On a simply connected manifold, it is well known that pseudo-K\"ahler metrics have holonomy contained in $\LieG{U}(p,q)$; Ricci-flatness is equivalent to holonomy being contained in $\SU(p,q)$, i.e. the existence of a closed complex volume form.

In the case of a nilpotent Lie algebra $\g$, the complex volume form is unique up to multiple; the fact that it is closed can be proved using the methods of~\cite{Salamon:ComplexStructures}, or directly as follows.

Let $\theta_1,\dotsc, \theta_n$ be a complex frame of vectors of type $(1,0)$, and let $\theta^1,\dotsc, \theta^n$ be the dual coframe of $(1,0)$-forms. Relative to the splitting $\g^{\C}=\g^{1,0}\oplus\g^{0,1}$, we have
\[\ad\overline\theta_k =\begin{pmatrix} f_k & 0 \\ * & *\end{pmatrix}, \quad k=1,\dotsc, n,\]
where $f_k\colon \g^{1,0}\to\g^{1,0}$ is nilpotent, hence trace-free. Therefore, for all $k$ we have
\[\overline\theta_k\hook d(\theta^1\wedge\dots\wedge\theta^n) = -\Tr(f_k)\theta^1\wedge\dots \wedge \theta^n=0,\]
implying that the complex volume form  $\theta^1\wedge\dots\wedge\theta^n$ is closed.
\end{proof}

\begin{proposition}
There is no nilpotent Sasaki-Einstein Lie algebra.
\end{proposition}
\begin{proof}
Let $\g$ be a nilpotent Lie algebra with a Sasaki-Einstein structure. We know that its center is spanned by $\xi$. By Proposition~\ref{prop:kahlerquotient}, the quotient $\g/\Span{\xi}$ is pseudo-K\"ahler and Einstein with positive scalar curvature. Since it is also nilpotent, it must be Ricci-flat by Lemma~\ref{lemma:finopartonsalamon}, which is absurd.
\end{proof}

Another link between Sasaki-Einstein and K\"ahler-Einstein geometry is given by the following:
\begin{proposition}[\protect{\cite[Corollary~11.1.8]{BoGa:SasakianBook}}]
Let $(\phi,\xi,\eta,g)$ be an almost contact pseudo-Riemannian metric structure on a manifold $M$ of dimension $2n+1$. The following are equivalent:
\begin{enumerate}
 \item $(\phi,\xi,\eta,g)$ is Sasaki-Einstein;
 \item the cone $(\R^+\times M,J,\omega)$ is pseudo-K\"ahler and Ricci-flat.
\end{enumerate}
\end{proposition}

\section{Einstein standard Lie algebras}
\label{sec:EinsteinStandard}
In this section we study the Einstein condition on standard Lie algebras $\g\rtimes\Span{e_0}$, without assuming the pseudo-Iwasawa condition (see \cite[Proposition~2.6]{ContiRossiSegnan:PseudoSasaki}). We write down the conditions that the induced metric $g$ and the derivation $D=\ad e_0$ must satisfy, generalizing the nilsoliton equation. In particular, the conditions are satisfied if $g$ is Ricci-flat and the symmetric part of $D$ is an appropriate multiple of the identity.

We then recall and generalize the construction of the Nikolayevsky and metric Nikolayevsky derivation (\cite{Nikolayevsky,ContidelBarcoRossi:Uniqueness}). We show that a nilpotent Lie algebra admits a standard Einstein extension with the symmetric part of $D$ equal to a multiple of the identity if and only if it is Ricci-flat and the metric Nikolayevsky derivation is nonzero. In this case, the extension is unique up to isometry.

Recall that given endomorphisms $f,g$ of $\g$, we have
\[g(f,g)=\Tr (fg^*)=\Tr (f(g^s-g^a)).\]
\begin{proposition}
\label{prop:generaleinstein}
Let $\g$ be a nilpotent Lie algebra with a pseudo-Riemannian metric $g$, $D$ a derivation and $\tau=\pm1$. Then the metric $\widetilde{g}=g+\tau e^0\otimes e^0$ on
$\widetilde{\g}=\g\rtimes_D\Span{e_0}$ is Einstein if and only if
\[\Ric=\tau\bigl(- \Tr ((D^s)^2) \id- \frac12[D,D^*] +(\Tr D)D^s\bigr), \qquad \Tr (\ad v\circ D^*)=0,\ v\in\g;\]
in this case, $\widetilde{\ric}= - \tau\Tr ((D^s)^2)\widetilde{g}$.
\end{proposition}
\begin{proof}
By \cite[Proposition~1.10]{ContiRossi:IndefiniteNilsolitons}, we have
\begin{gather*}
\widetilde{\ric}(v,w)=\ric(v,w)+\tau\widetilde{g}(\frac12 [D,D^*](v),w)-\tau(\Tr D)\widetilde{g}(D^s(v),w)\\
\widetilde{\ric}(v,e_0)=\frac12\widetilde{g}(\ad v, D)\\
\begin{split}
\widetilde{\ric}(e_0,e_0)&=-\frac12\widetilde{g}(D,D)-\frac12 \Tr D^2=-
\frac12 \Tr D(D^s-D^a)-\frac12 \Tr D(D^s+D^a)\\
&=- \Tr DD^s=-\Tr (D^s)^2.
\end{split}
\end{gather*}
Thus, the Einstein condition $\widetilde{\Ric}=\lambda \id$ holds if and only if
\[
\lambda \id=\Ric + \frac12\tau[D,D^*]  -\tau(\Tr D)D^s,\qquad
\Tr (\ad v\circ D^*)=0,\qquad
\lambda = -\tau \Tr (D^s)^2.\qedhere
\]
\end{proof}
\begin{remark}
If $h=-g$, then $h,g$ have the same Ricci tensor and opposite Ricci operators; the operators $D\mapsto D^*$ and $D\mapsto D^s$ are identical. Therefore, if $g$ satisfies
\[\Ric^g=\tau\bigl(- \Tr ((D^s)^2) \id- \frac12[D,D^*] +(\Tr D)D^s\bigr), \qquad \Tr (\ad v\circ D^*)=0,\ v\in\g,\]
then
\[\Ric^h=(-\tau)\bigl(- \Tr ((D^s)^2) \id- \frac12[D,D^*] +(\Tr D)D^s\bigr), \qquad \Tr (\ad v\circ D^*)=0,\ v\in\g.\]
This amounts to the fact that $g+\tau e^0\otimes e^0$ is Einstein if and only if so is $h-\tau e^0\otimes e^0$.
\end{remark}
\begin{remark}
 We can write
 \[[D,D^*]=[D^a+D^s,-D^a+D^s]=2[D^a,D^s].\]
\end{remark}
It turns out that the condition that $\Tr(\ad v\circ D^*)$ vanish can be eschewed under a suitable assumption on the eigenvalues of $D$:
\begin{corollary}
\label{cor:TrNotEigenvalue}
Let $\g$ be a nilpotent Lie algebra with a pseudo-Riemannian metric $g$, $D$ a derivation such that $-\Tr D$ is not an eigenvalue of $D$ and $\tau=\pm1$. Then the metric $\widetilde{g}=g+\tau e^0\otimes e^0$ on
$\widetilde{\g}=\g\rtimes_D\Span{e_0}$ is Einstein if and only if
\begin{equation}
\label{eqn:generalizednilsoliton}
\Ric=\tau\bigl(- \Tr ((D^s)^2) \id- \frac12[D,D^*] +(\Tr D)D^s\bigr);
\end{equation}
in this case, $\widetilde{\ric}= - \tau \Tr ((D^s)^2)\widetilde{g}$.
\end{corollary}
\begin{proof}
One direction follows from Proposition~\ref{prop:generaleinstein}. For the other direction, assume that $f=(\Tr D)\id +D$ is invertible and~\eqref{eqn:generalizednilsoliton} holds. Since $\ad v$ is a derivation,
\begin{multline*}
0=\Tr(\ad v\circ\Ric)=
 - \Tr ((D^s)^2) \Tr \ad v- \frac12\Tr([D,D^*]\circ \ad v) +(\Tr D)\Tr (\ad v\circ D^s)\\
 = -\frac12\Tr([\ad v,D]\circ D^*) +\frac12(\Tr D)\Tr(\ad v\circ (D+D^*))\\
 = \frac12\Tr(\ad Dv\circ D^*) +\frac12(\Tr D)\Tr(\ad v\circ D^*)
 =\frac12\Tr(\ad(f(v))\circ D^*),
\end{multline*}
where we have used $\Tr (\ad v\circ D)=0$ (see e.g. \cite[Chapter~1, Section~5.5]{Bourbaki:LieGropuCh123}).
Since $f$ is invertible, this implies that $\Tr (\ad w\circ D^*)=0$ for all $w$, so $\widetilde{g}$ is Einstein by Proposition~\ref{prop:generaleinstein}.
\end{proof}
Proposition~\ref{prop:generaleinstein} and Corollary~\ref{cor:TrNotEigenvalue} generalize a similar result of~\cite{ContiRossi:IndefiniteNilsolitons}, where the derivation $D$ was assumed to be symmetric. The resulting standard extensions take the form $\widetilde{\g}=\g\rtimes_D\Span{e_0}$, with $D$ symmetric; such a standard decomposition is said to be of \emph{pseudo-Iwasawa type}.

\begin{example}
Fix the  Lie algebra $\g=(0,0,e^{12},0)$, which is the direct sum of the Heisenberg Lie algebra and $\R$; the notation, inspired by~\cite{Salamon:ComplexStructures}, means that $\g^*$ has a fixed basis of $1$-forms $e^1,e^2,e^3,e^4$ with $de^3=e^1\wedge e^2$ and the other forms closed. The  metric $g= e^1\odot e^2+ e^3\odot e^4$,  has Ricci operator equal to
\[
   \Ric = \begin{pmatrix}
      0&0&0&0\\
      0&0&0&0\\
      0&0&0&-\frac{1}{2}\\
      0&0&0&0
   \end{pmatrix}.
\]
Consider the derivation
   \[D=
   \begin{pmatrix}
      -\frac{\mu}{4} & \lambda & 0 & 0 \\
      -\frac{\mu^2}{8 \lambda} & -\frac{\mu}{4} & 0 & 0 \\
      0 & 0 & -\frac{\mu}{2} & -\frac{1}{3 \mu\tau} \\
      0 & 0 & 0 & \mu
   \end{pmatrix},\]
   where $\lambda$ and $\mu$ are nonzero parameters.    Then equation~\eqref{eqn:generalizednilsoliton} is satisfied for any choice of $\tau=\pm1$. In this case
   \[
   D^s=\begin{pmatrix}
      -\frac{\mu}{4} & \lambda & 0 & 0 \\
      -\frac{\mu^2}{8 \lambda} & -\frac{\mu}{4} & 0 & 0 \\
      0 & 0 & \frac{\mu}{4} & -\frac{1}{3 \mu\tau} \\
      0 & 0 & 0 & \frac{\mu}{4}
   \end{pmatrix},
   \]
   hence $\Tr(D^s)=0$ and $\Tr((D^s)^2)=0$.

In order to obtain a standard Einstein metric, it is sufficient, thanks to Corollary~\ref{cor:TrNotEigenvalue}, to show that $\Tr D=0$ is not an eigenvalue. Since  $\mu$ is assumed not to be zero, $D$ is not singular and $0$ cannot be an eigenvalue. For $\tau=1$, we obtain a two-parameter family of Ricci-flat solvmanifolds of signature $(3,2)$; for $\tau=-1$ we obtain another family which corresponds to reversing the overall sign of the metric and applying the isomorphism $ e_2\mapsto -e_2$, $e_3\mapsto -e_3$.

Notice that the resulting standard  Lie algebra $\widetilde{\g}=\g\rtimes\Span{e_0}$ has derived algebra equal to $\g$, because $D$ is surjective. Therefore, the standard decomposition is unique. In addition, it is not possible to use Proposition~\ref{prop:pseudoAzencottWilson} to obtain an isometric standard Lie algebra of pseudo-Iwasawa type because $D$ and $D^s$ do not commute.
\end{example}

\begin{example}
In this example we apply Corollary~\ref{cor:TrNotEigenvalue} to a nilpotent Lie  algebra of step greater than $2$, and obtain an Einstein solvmanifold with nonzero scalar curvature.
Fix the $4$-dimensional $3$-step nilpotent Lie algebra $\g=(0,0,e^{12},e^{13})$ with metric $g= e^1\odot e^3+\frac{1}{2}e^2\otimes e^2-e^{4}\otimes e^4$. Its  Ricci operator equals
\[
   \Ric = \begin{pmatrix}
       -\frac{1}{2} & 0 & 0 & 0 \\
       0 & 0 & 0 & 0 \\
       0 & 0 & -\frac{1}{2} & 0 \\
       0 & 0 & 0 & \frac{1}{2}
   \end{pmatrix}.
\]
Consider the one-parameter family of  derivations
   \[D=
   \begin{pmatrix}
 -\sqrt{\frac{3}{8}} & 0 & 0 & 0 \\
 0 & \sqrt{\frac{3}{2}} & 0 & 0 \\
 \mu & 0 & \sqrt{\frac{3}{8}} & 0 \\
 0 & 1 & 0 & 0
   \end{pmatrix},\]
   where $\mu\in\R$.    Then equation~\eqref{eqn:generalizednilsoliton} is satisfied with $\tau=1$. In this case
   \[
   D^s=\begin{pmatrix}
    0 & 0 & 0 & 0 \\
    0 & \sqrt{\frac{3}{2}} & 0 & -1 \\
    \mu & 0 & 0 & 0 \\
    0 & \frac{1}{2} & 0 & 0
   \end{pmatrix},
   \]
hence $\Tr D=\Tr (D^s)=\sqrt{\frac{3}{2}}$ and $\Tr((D^s)^2)=\frac{1}{2}$. Note that $-\Tr D$ is not an eigenvalue, then  by Corollary~\ref{cor:TrNotEigenvalue} for any choice of $\mu$ the Lie algebra
$\widetilde{\g}=\g\rtimes\Span{e_0}$ has a standard Einstein metric $\widetilde{g}= g+e^0\otimes e^0$. Even if  $D$ is not surjective, the resulting standard Einstein Lie algebra has derived algebra equal to $\g$, so the standard decomposition is unique. In addition, it is not possible to obtain an isometric standard Lie algebra of pseudo-Iwasawa type using either \cite[Proposition~1.19]{ContiRossi:IndefiniteNilsolitons} (because $D^*$ is not a derivation) or Proposition~\ref{prop:pseudoAzencottWilson} (because $D$ and $D^s$ do not commute).
\end{example}

As a particular case of Proposition~\ref{prop:generaleinstein}, consider solutions of~\eqref{eqn:generalizednilsoliton} such that $D^s=a\id$.
The case $a=0$ corresponds to a standard extension by a skew-symmetric derivation of a Ricci-flat metric, which by Proposition~\ref{prop:pseudoAzencottWilson} yields a Ricci-flat metric isometric to a product with a line.

In the case $a\neq0$, we have that $D$ is a derivation in the Lie algebra
\[\co(r,s)=\so(r,s)\oplus\Span{\id},\]
where $(r,s)$ is the signature of $g$, and the inclusion $\co(r,s)\subset\gl(\g)$ is determined by fixing an orthonormal frame. Additionally, $D$ has nonzero trace. This implies that the \emph{metric Nikolayevsky} derivation $N$ is nonzero. We proceed to recall the construction of $N$, giving a slight generalization for use in later sections. For the proof, we refer to~\cite{Nikolayevsky} and \cite[Theorem~4.9]{ContidelBarcoRossi:Uniqueness}.
\begin{proposition}
\label{prop:generalizednik}
Let $\lie h$ be an algebraic subalgebra of $\gl(m,\R)$. There exists a semisimple derivation  $N$ in $\lie h\cap\Der\g$ such that
\[\Tr(N\psi)=\Tr\psi, \qquad \psi\in \lie h\cap\Der\g.\]
The derivation $N$ is unique up to automorphisms of $\lie h$.
\end{proposition}
For $\lie h=\gl(m,\R)$ the derivation $N$ of Proposition~\ref{prop:generalizednik} corresponds to the pre-Einstein or Nikolayevsky derivation introduced in~\cite{Nikolayevsky}; accordingly, we will refer to the derivation $N$ of Proposition~\ref{prop:generalizednik} as the $\lie h$-Nikolayevsky derivation. For $\lie h=\co(r,s)$, the $\lie h$-Nikolayevsky derivation is the metric Nikolayevsky derivation introduced in~\cite{ContidelBarcoRossi:Uniqueness}.

Notice that the $\lie h$-Nikolayevsky derivation is zero if and only if all derivations in $\lie h$ are traceless (i.e. $\lie h$ is contained in $\Sl(m,\R)$).  In particular, we see that that there is derivation with $D^s=\id$ if and only if the metric Nikolayevsky derivation is nonzero.

In later sections, we will consider Lie algebras with an almost pseudo-Hermitian structure and  use the $\cu(p,q)$-Nikolayevsky derivation, where
\[\cu(p,q)=\lie u(p,q)\oplus \Span{\id}.\]
Like the Nikolayevsky and the metric Nikolayevsky derivation, the $\cu(p,q)$-Nikolayevsky derivation turns out to have rational eigenvalues:
\begin{proposition}
\label{prop:rationaleigenvalues}
Let $\g$ be a Lie algebra with an almost pseudo-Hermitian structure. Then the $\cu(p,q)$-Nikolayevsky derivation of $\g$ has rational eigenvalues.
\end{proposition}
\begin{proof}
The proof follows~\cite{Nikolayevsky} and \cite[Theorem~4.9]{ContidelBarcoRossi:Uniqueness}. We can characterize elements of $\cu(p,q)$ as elements of $\co(2p,2q)$ that commute with the complex structure $J$.

If $N$ is the $\cu(p,q)$-Nikolayevsky, let $\g^\C=\bigoplus \lie b_t$ be the decomposition into eigenspaces and let  $\pi_t\colon\g^\C\to \lie b_t$ denote the projections. Define
\[\lie n=\left\{\sum \nu_t \pi_t\st \sum \nu_t \pi_t\in(\Der\g\cap \co(2p,2q))^\C\right\}.\]
Since $N$ commutes with $J$, each $\lie b_t$ is $J$-invariant. Therefore, $J$ commutes with projections, and we can write
\[\lie n=\left\{\sum \nu_t \pi_t\st \sum \nu_t \pi_t\in(\Der\g\cap \lie h)^\C\right\}.\]
One can now proceed as in \cite[Theorem~4.9]{ContidelBarcoRossi:Uniqueness} and show that $N$ is the unique element of $\lie n$ such that $\Tr (N\psi)=\Tr(\psi)$ for all $\psi\in\lie n$, and its coefficients $\nu_t$ are rational numbers.
\end{proof}

\begin{lemma}
\label{lemma:equivalentextensions}
Let $H$ be an algebraic subgroup of $\SO(r,s)$ with Lie algebra $\lie h$ and let $\g$ be a nilpotent Lie algebra with a $H$-structure. If $D,D'$ are two elements of $(\lie h\oplus\Span{\id})\cap\Der\g$ with the same trace, then the $H$-structures on $\g\rtimes_D\Span{e_0}$ and $\g\rtimes_{D'}\Span{e_0}$ are equivalent.
\end{lemma}
\begin{proof}
The Lie algebra $\lie{f} =(\lie h\oplus\Span{\id})\cap\Der\g$ is algebraic. Observe that that two commuting derivations of $\lie{f}$ with the same trace determine equivalent extensions by Proposition~\ref{prop:pseudoAzencottWilson}, as their difference is in $\lie h\cap\so(r,s)$. We will use this fact repeatedly.

Denote by $\lie r$ the radical of $\lie{f}$. By~\cite{Chevalley}, the fact that $\lie{f}$ is algebraic implies that  $\lie r$ is also algebraic, and we can write $\lie r=\lie n\rtimes\lie a$, where $\lie a$ is an abelian Lie algebra consisting of semisimple elements and $\lie n$ is the nilradical. Since $\lie a$ is abelian, any two derivations in $\lie a$ with the same trace determine isometric extensions. Thus, we only need to show that for any $D\in\lie{f}$ there is an element of $\lie a$ determining an equivalent extension.

Since $\lie{f}$ is algebraic, we can write $D=D_{ss}+D_n$, where $D_{ss}$ is semisimple, $D_n$ is nilpotent, and $[D_{ss},D_n]=0$. Since $D_n$ has trace zero, $D$ and $D_{ss}$ determine isometric extensions. Since $D_{ss}$  is semisimple, so are
\[\ad D_{ss}\colon \lie{f}\to\lie{f}, \qquad\ad D_{ss}\colon \lie{f}_0\to\lie{f}_0,\]
where we have set $\lie{f}_0=\lie{f}\cap\so(r,s)$. We can choose a decomposition
\[\lie{f}=\lie r\oplus W,\]
where $W$ is contained in $\lie h$ and $\ad D_{ss}$-invariant. Indeed, it suffices to choose for $W$ an $\ad D_{ss}$-invariant complement of $\lie{f}_0\cap\lie r$ in $\lie{f}_0$.

Accordingly, write $D_{ss}=D_{\lie r}+D_{W}$. Then
\[[D_{ss},D_{W}]=[D_{\lie r},D_{W}];\]
the left-hand side belongs to the $\ad D_{ss}$-invariant space $W$, and the right-hand side to the ideal $\lie r$, so both must vanish.

Therefore, $D_{ss}$ and $D_{\lie r}$ are commuting derivations with the same trace, and they detemine equivalent extensions.

Using the Jordan decomposition in the algebraic Lie algebra $\lie r$, we see that $D_{\lie r}$ determines, up to equivalence, the same standard extension as its semisimple part. On the other hand, the latter is conjugate in $\lie r$ to an element of $\lie a$ by \cite[Section~19.3]{Humphreys:LinearAlgebraicGroups}. The conjugation is realized by an element of the Lie group with Lie algebra $\lie{f}$ which can be assumed to have determinant one, and therefore by an element of $H$.
\end{proof}

\begin{theorem}
\label{thm:metricnikextension}
Let $\g$ be a nilpotent Lie algebra with a pseudo-Riemannian metric $g$  such that the metric Nikolayevsky derivation $N$ is nonzero. Then $g$ is Ricci-flat and $\g$ has an Einstein standard extension $\g\rtimes_N\Span{e_0}$.

Conversely, suppose $\g$ is a nilpotent Lie algebra with a pseudo-Riemannian metric $g$ and an Einstein standard extension with $D^s=a\id$. Then $g$ is Ricci-flat and, up to a scaling factor, the extension is isometric to either $\g\oplus\R$ or $\g\rtimes_N\Span{e_0}$ according to whether $a$ is zero or not.
\end{theorem}
\begin{proof}
Let $D$ be a multiple of $N$ such that $D^s=\id$. Every metric of the form $e^tg$ can be written as $g(\exp (tD)\cdot, \exp (tD)\cdot)$, i.e. it is related to $g$ by an isomorphism. The Ricci tensor transforms accordingly; however, the Ricci tensor of $e^tg$ coincides with that of $g$, and this forces it to be zero. Then $[D,D^*]=[D,2\id-D]=0$
and~\eqref{eqn:generalizednilsoliton} holds. In addition,
\[\Tr (\ad v\circ D^*)=\Tr (\ad v\circ (2\id-D))=\Tr( 2\ad v - \ad v\circ D)=0,\]
where $\ad v$ and $\ad v\circ D$ are traceless because $\g$ is nilpotent and $D$ is a multiple of the Nikolayevsky derivation. Thus, Proposition~\ref{prop:generaleinstein} implies that $\g\rtimes_D \Span{e_0}$ is Einstein.

We claim that replacing $D$ with a nonzero  multiple, say $D'=kD$, has the effect of giving the same standard extension up to isometry and rescaling. Indeed, observe that $\{\exp tD\}$ acts on the metric $g$ by rescaling while leaving $D$ unchanged. This means that the $\widetilde{g}=g+e^0\otimes e^0$ and $\widetilde{g}'=k^2g+e^0\otimes e^0$ are isometric metrics on
 $\widetilde{\g}=\g\rtimes_D\Span{e_0}$. Setting $e_0'=ke_0$, we can write $\widetilde{g}'=k^2(g+(e^0)'\otimes (e^0)')$, and $\widetilde{\g}=\g\rtimes_{D'}\Span{e_0}$.

Now suppose that $\g$ has a standard Einstein extension with $D^s=a\id$. In this case, if $\g$ has dimension $m$ and $D^s=a\id$, then $[D,D^*]=2[D^a,D^s]=0$ and~\eqref{eqn:generalizednilsoliton} becomes
\begin{equation*}\label{eqn:generalizednilsolitonforDinco}
\Ric=\tau (-a^2m\id+ma^2\id)=0.
\end{equation*}
If $a=0$, $D$ is skew-symmetric; by Proposition~\ref{prop:pseudoAzencottWilson}, we can assume  $D=0$ up to isometry, obtaining a direct product $\g\times\R$.

If $a\neq0$, $D$ has nonzero trace and the metric Nikolayevsky derivation $N$ is nonzero, so it too has nonzero trace. We already observed that rescaling $N$ yields an isometric extension up to isometry. Therefore, we can assume that $D$ and $N$ have the same trace and conclude by Lemma~\ref{lemma:equivalentextensions}.
\end{proof}
\begin{remark}
\label{remark:semidirectgroups}
Geometrically, we can describe the metric of Theorem~\ref{thm:metricnikextension} as follows. Let $G$ be the simply connected Lie group with Lie algebra $\g$. We can exponentiate $N$ to a one-parameter group of automorphisms $\{f_t\}=\{\exp tN\}\subset\Aut\g$, which determines a one-parameter group of automorphisms $\{\phi_t\}$ in $\Aut G$. The semidirect product $G\rtimes_{\phi_t}\R$ has Lie algebra $\g\rtimes_N\R$, and a left-invariant metric $g$ on $G$ induces a left-invariant metric $(\exp tN)g+dt^2$. Since skew-symmetric elements act trivially on the metric and $N^s$ is a multiple of the identity, the metric takes the form of a warped product,
\[(\exp tN^s)g+dt^2= e^{2kt}g+dt^2, \quad k=\frac{\Tr N}{\dim\g}.\]
The fact that a metric of this form is Einstein follows directly from $g$ being Ricci-flat (see \cite[\S~9.109]{Besse}).
\end{remark}

\section{$\lie z$-standard Sasaki-Einstein Lie algebras}
\label{sec:InvariantONeill}
In this section we study the Ricci curvature of $\lie z$-standard Sasaki-Lie algebras, characterizing the Einstein metrics in terms of their K\"ahler reduction. Recall from Section~\ref{sec:preliminaries} that a $\lie z$-standard Sasaki-Lie algebra is a Lie algebra $\widetilde{\g}$ carrying both a Sasaki structure  $(\phi, \xi, \eta,g)$ and a standard decomposition of the form $\widetilde{\g}=\g\rtimes\Span{e_0}$ such that $\phi(e_0)$ lies in the center $\lie z$ of $\g$.

Let $\g$ be a central extension of a nilpotent Lie algebra $\widecheck{\g}$, i.e.
\[0\to \R^k\to \g\to\widecheck{\g}\to 0.\]
As vector spaces, $\g=\widecheck{\g}\oplus\R^k$. Let $\{e_s\}$ be a basis of $\R^k$; the elements $\{e^s\}$ of the dual basis can be viewed as elements of $\g^*$, and the Lie algebra structure of $\g$ is entirely determined by $\widecheck{\g}$ and the exterior derivatives $\{de^s\}$. Explicitly,
\[[v,w]_\g=[v,w]_{\widecheck{\g}}-\sum_s de^s(v,w)e_s, \quad v,w\in\g.\]
\begin{lemma}
\label{lemma:oneillcentral}
Let $\widecheck{\g}$ be a nilpotent Lie algebra with a metric $\widecheck{g}$; on the central extension $\g=\widecheck{\g}\oplus\R^k$, fix a metric of the form
\[g=\widecheck{g}+\sum_s \epsilon_s e^s\otimes e^s,\quad \epsilon_s=\pm1.\]
Then, for $v,w\in\widecheck{\g}$, the Ricci tensors of $g$ and $\widecheck{g}$ are related by
\begin{gather*}
\ric(v,w)=\widecheck{\ric}(v,w)-\frac12\sum_s \epsilon_s g(v\hook de^s, w\hook de^s),\\
\ric(v,e_s)=\frac12\epsilon_sg(dv^\flat, de^s),\qquad \ric(e_s,e_t)=\frac12\epsilon_s\epsilon_t g(de^s,de^t).
\end{gather*}
\end{lemma}
\begin{proof}
By construction,
$\ad v=\adcheck v-\sum_s v\hook de^s\otimes e_s$.
For one-forms $\alpha$ on $\widecheck{\g}$, zero-extended to $\g$, we have
$d\alpha=\widecheck{d}\alpha$.
We use the fact that the musical isomorphisms relative to $g$ and $\widecheck{g}$ are compatible, so using~\eqref{eqn:riccipeaceful} we obtain
\begin{equation*}
\begin{split}
\ric(v,w)&=\frac12 g(dv^\flat, dw^\flat)-\frac12g(\ad v,\ad w)\\
&=\frac12 g(\widecheck{d}v^\flat, \widecheck{d}w^\flat)-\frac12g(\adcheck v,\adcheck w)-\frac{1}{2}g(\sum_sv\hook de^s\otimes e_s,\sum_\ell w\hook de^\ell\otimes e_\ell)\\
&=\widecheck{\ric}(v,w)-\frac12\sum_s \epsilon_s g(v\hook de^s, w\hook de^s).\qedhere
\end{split}
\end{equation*}
\end{proof}

\begin{lemma}
\label{lemma:Riccig}
The Ricci tensor of the metric on $\g$ constructed in Proposition~\ref{prop:constructive} is
\begin{gather*}
\Ric(v)=-2(\tau (\widecheck{D}^s)^2+\id)v, \quad v\in\Span{b,\xi}^\perp,\\
\Ric(b)=\tau \Tr((\widecheck{D}^s)^2)b- (\Tr \widecheck{D})\xi, \qquad \Ric(\xi)=(2n-2)\xi-\tau(\Tr\widecheck{D})b.
\end{gather*}
where $\dim \g=2n$.
\end{lemma}
\begin{proof}
Since $\widecheck{\g}$ is pseudo-K\"ahler and nilpotent, $\widecheck{\ric}$ is zero by Lemma~\ref{lemma:finopartonsalamon}. By Lemma~\ref{lemma:oneillcentral}, we have
\begin{equation*}
\begin{split}
\ric(v,w)&=-\frac12\tau g(v\hook db^\flat, w\hook db^\flat)-\frac12 g(v\hook d\eta,w\hook d\eta)\\
&=-\frac12\tau g(\widecheck{D}^s(v)\hook d\eta, \widecheck{D}^s(w)\hook d\eta)-\frac12 g(v\hook d\eta,w\hook d\eta)\\
&=-2\tau g(JD^s(v),JD^s(w))-2g(Jv,Jw)=-2\tau g(D^s(v),D^s(w))-2g(v,w).
\end{split}
\end{equation*}
Then
\begin{align*}
\ric(v,b)&=\frac12\tau g(dv^\flat, db^\flat)=\frac12\tau g(dv^\flat, \widecheck{D}\omega), &
\ric(v,\xi)&=\frac12\tau g(dv^\flat, d\eta)=\tau g(dv^\flat,\omega),\\
\ric(b,b)&=\frac12 g(db^\flat,db^\flat)=\frac12g(\widecheck{D}\omega,\widecheck{D}\omega), &
\ric(b,\xi)&=\frac12 g(db^\flat,d\eta)= g(\widecheck{D}\omega,\omega),\\
\ric(\xi,\xi)&=\frac12\tau g(d\eta,d\eta)=2g(\omega,\omega).
\end{align*}
We can simplify these formulae by observing that
\begin{multline*}
\widecheck{D}\omega(x,y)=-\omega(\widecheck{D} x, y)-\omega(x,\widecheck{D} y)=-\widecheck{g}(\widecheck{D}x,Jy)-\widecheck{g}(x,J\widecheck{D}y)\\
=-\widecheck{g}(x,(J\widecheck{D}+\widecheck{D}^*J)y)=-\widecheck{g}(x,(\widecheck{D}+\widecheck{D}^*)J y),
\end{multline*}
so we can view $\widecheck{D}\omega$ as a $(1,1)$ tensor $(\widecheck{D}\omega)^\sharp=-(\widecheck{D}+\widecheck{D}^*)J$. Similarly, we have $\omega^\sharp = J$. Then
\begin{align*}
g(\omega,\omega)&=\frac12 g(J,J)=n-1,\\
g(\omega, \widecheck{D}\omega)&=\frac12g(J,-(\widecheck{D}+\widecheck{D}^*)J)=\frac12\Tr((\widecheck{D}+\widecheck{D}^*)J^2)=-\frac12\Tr(\widecheck{D}+\widecheck{D}^*)=-\Tr\widecheck{D}^s=-\Tr\widecheck{D},\\
g(\widecheck{D}\omega,\widecheck{D}\omega)&=\frac12 g((\widecheck{D}+\widecheck{D}^*)J,(\widecheck{D}+\widecheck{D}^*)J)
 =\frac12\Tr((\widecheck{D}+\widecheck{D}^*)^2)=2\Tr (\widecheck{D}^s)^2.
\end{align*}
Finally, observe that $\omega$ and $\widecheck{D}\omega$ are $d^*$-closed, so (since $\widecheck{\g}$ is unimodular),
\[g(dv^\flat,\omega)=g(v^\flat,d^*\omega)=0, \qquad g(dv^\flat,\widecheck{D}\omega )=g(v^\flat, d^*\widecheck{D}\omega)=0.\]

Summing up,
\begin{align*}
\ric(v,w)&=-2\tau g(\widecheck{D}^s(v),\widecheck{D}^s(w))-2g(v,w), &
\ric(v,b)&=0,\\
\ric(v,\xi)&=0, &
\ric(b,b)&=\Tr ((\widecheck{D}^s)^2),\\
\ric(b,\xi)&=- \Tr\widecheck{D}, &
\ric(\xi,\xi)&=(2n-2).\qedhere
\end{align*}
\end{proof}

\begin{lemma}
\label{lemma:einsteiniff}
With the hypothesis of Proposition~\ref{prop:constructive}, the metric $\widetilde{g}=g+\tau e^0\otimes e^0$ on $\widetilde{\g}=\g\rtimes_D\Span{e_0}$ is Einstein if and only if
\[\tau=-1, \qquad \widecheck{D}^s=\pm \id, \qquad h=\pm 2.\]
\end{lemma}
\begin{proof}
By Proposition~\ref{prop:generaleinstein}, $\widetilde{g}$ is Einstein if and only if
\[\Ric=\tau\bigl(- \Tr ((D^s)^2) \id+ [D^s,D^a] +(\Tr D)D^s\bigr), \qquad \Tr (\ad v\circ D^*)=0,\quad v\in\g.\]
We have
\begin{align*}
D&=\begin{pmatrix} \widecheck{D} & 0 & 0 \\ 0 & h & 0 \\ 0 & -2\tau & 0\end{pmatrix},
&
D^*&=\begin{pmatrix} \widecheck{D}^* & 0 & 0 \\ 0 & h & -2 \\ 0 & 0 & 0\end{pmatrix},\\
D^s&=\begin{pmatrix} \widecheck{D}^s & 0 & 0 \\ 0 & h & -1 \\ 0 & -\tau & 0\end{pmatrix},
&
D^a&=\begin{pmatrix} \widecheck{D}^a & 0 & 0 \\ 0 & 0 & 1 \\ 0 & -\tau & 0\end{pmatrix}.
\end{align*}
So
\[[D^s,D^a]=\begin{pmatrix} h\widecheck{D}^s-2(\widecheck{D}^s)^2 &0 &0 \\ 0 &  2\tau & h\\ 0 & h\tau  & -2\tau\end{pmatrix}.\]

Multiplying by $\tau$ each side of~\eqref{eqn:generalizednilsoliton}
and using Lemma~\ref{lemma:Riccig}, we get
\begin{multline*}
\begin{pmatrix}
-2(\widecheck{D}^s)^2-2\tau \id & 0 & 0 \\
0 & \Tr((\widecheck{D}^s)^2) & - (\Tr\widecheck{D})\\
0 & - \tau \Tr \widecheck{D} & \tau(2n-2)
\end{pmatrix}
=- (\Tr ((\widecheck{D}^s)^2)+h^2+2\tau) \begin{pmatrix} \id& 0 & 0 \\ 0 & 1 & 0 \\ 0 & 0 & 1 \end{pmatrix}\\[4pt]
+
\begin{pmatrix}  h\widecheck{D}^s-2 (\widecheck{D}^s)^2 &0 &0 \\ 0 &  2\tau & h \\ 0 & h\tau  & -2\tau \end{pmatrix}
+(\Tr \widecheck{D}^s+h)
\begin{pmatrix} \widecheck{D}^s & 0 & 0 \\ 0 &  h & -1 \\ 0 & -\tau & 0\end{pmatrix},
\end{multline*}
i.e.
\begin{align*}
(\Tr ((\widecheck{D}^s)^2)+h^2)\id&=(\Tr \widecheck{D}^s+2h) \widecheck{D}^s,\\
2\Tr ((\widecheck{D}^s)^2)&=(\Tr \widecheck{D}^s)h,\\
\tau(2n+2)&=- (\Tr ((\widecheck{D}^s)^2)+h^2).
\end{align*}
If this system of equations holds, $\widecheck{D}^s$ is a multiple of the identity; setting $\Tr \widecheck{D}^s=\lambda$, so that $\Tr ((\widecheck{D}^s)^2)= \dfrac{\lambda^2}{2n-2}$, we get
\[\tau=-1, \qquad h=\frac{2\lambda}{2n-2}, \qquad \lambda=\pm(2n-2).\]
So the system holds if and only if $\widecheck{D}^s=\pm \id$ and $h=\pm 2$. This condition also implies  $\Tr(\ad v\circ D^*)=0$ because $\g$ is unimodular and $\Tr (\ad v\circ D)=0$ by \cite[Chapter~1, Section~5.5]{Bourbaki:LieGropuCh123}, proving the equivalence in the statement.
\end{proof}
\begin{remark}
As observed in \cite[Remark~5.2]{ContiRossiSegnan:PseudoSasaki}, changing the sign of $h$, $\widecheck{D}$, $e_0$ and $b$ yields an isometric metric. Therefore, we will only consider the case $h=2$ and $\widecheck{D}^s=\id$.
\end{remark}

The construction of Proposition~\ref{prop:constructive} can be specialized to the Sasaki-Einstein case as follows:
\begin{proposition}
\label{prop:constructse}
Let $(\widecheck{\g},J,\omega)$ be a pseudo-K\"ahler nilpotent Lie algebra and let $\widecheck{D}$ be a derivation of $\widecheck{\g}$ with $\widecheck{D}^s=\id$ and commuting with $J$. If  $\g=\widecheck{\g}\oplus\Span{b,\xi}$ is the central extension of $\widecheck{\g}$ characterized by $d\xi^*=2\omega=db^*$, where $\{b^*,\xi^*\}$ is the basis dual to $\Span{b,\xi}$, with the metric $g=\widecheck{g} -b^*\otimes b^*+\xi^*\otimes\xi^*$,
then the semidirect product $\widetilde{\g}=\g\rtimes \Span{e_0}$, where
\[[e_0,x]=\widecheck{D}x, \qquad [e_0,b]=2b+2 \xi, \qquad [e_0,\xi]=0\]
has a Sasaki-Einstein structure $(\phi,\eta,\xi,\widetilde{g})$ given by
\[\widetilde{g}=g-e^0\otimes e^0, \qquad  \phi(w)=J(w)- g(b,w)e_0, \qquad  \phi(e_0)=-b, \quad w\in\g.\]
\end{proposition}
\begin{proof}
We have $\widecheck{D}\omega=\widecheck{D}^s\omega=-2\omega$; applying Proposition~\ref{prop:constructive} with $h=2$ and $\tau=-1$ we obtain a Sasaki extension as in the statement, which is Einstein by Lemma~\ref{lemma:einsteiniff}.
\end{proof}

Proposition~\ref{prop:constructse} has a K\"ahler analogue:
\begin{corollary}
\label{cor:constructke}
Let $(\widecheck{\g},J,\omega)$ be a pseudo-K\"ahler nilpotent Lie algebra with nonzero metric Nikolayevsky derivation, and let $\widecheck{D}$ be a derivation of $\widecheck{\g}$ with $\widecheck{D}^s=\id$. If  $\lie{k}=\widecheck{\g}\oplus\Span{b}$ is the central extension of $\widecheck{\g}$ characterized by $db^*=2\omega$, where $\{b^*\}$ is the basis dual to $\Span{b}$, with the metric $g=\widecheck{g} -b^*\otimes b^*$,
then the semidirect product $\widetilde{\lie{k}}=\lie{k}\rtimes \Span{e_0}$, where
\[[e_0,x]=\widecheck{D}x, \qquad [e_0,b]=2b\]
has a pseudo-K\"ahler-Einstein structure $(\widetilde{\lie{k}}, \widetilde{J}, \widetilde{\omega})$ given by
\[\widetilde{g}=g-e^0\otimes e^0, \qquad  \widetilde{J}(w)=J(w)- g(b,w)e_0, \qquad  \widetilde{J}(e_0)=-b, \quad w\in\g,\]
with $\widetilde{\ric}=(2n+2)\widetilde{g}$, with $2n$ the dimension of $\widetilde{\lie{k}}$.
\end{corollary}
\begin{proof}
Take the Lie algebra constructed in Proposition~\ref{prop:constructse} and take the quotient by $\xi$. Then by Proposition~\ref{prop:kahlerquotient} it is K\"ahler-Einstein with $\widetilde{\ric}=(2n+2)\widetilde{g}$.
\end{proof}
\begin{remark}
Arguing as in Remark~\ref{remark:semidirectgroups}, it follows that the pseudo-K\"ahler-Einstein metric constructed in Corollary~\ref{cor:constructke} has the form~\eqref{eqn:flippedberardbergery}.
\end{remark}

\begin{remark}\label{rem:PKEnotPIwasawa}
If the Lie algebra $\widecheck{\g}$ is not abelian, then Corollary~\ref{cor:constructke} produces pseudo-K\"ahler-Einstein rank-one extension which are not pseudo-Iwasawa, unlike the method presented in~\cite{Rossi:NewSpecialPseudoEinstein}, where one constructs  pseudo-K\"ahler-Einstein rank-one extensions of pseudo-Iwasawa-type.

Indeed, the derivation $\widecheck{D}=\ad e_0$ of Corollary~\ref{cor:constructke} is self-adjoint with respect to the metric if and only if $\widecheck{D}^s=\frac{1}{2}(D+D^*)$ is a derivation, but since $\widecheck{D}^s= \id$, this happens only if the identity is a derivation, i.e. if $\widecheck{\g}$ is an abelian Lie algebra.
\end{remark}

\begin{example}
\label{ex:extendabelian}
Let $\widecheck{\g}=\R^{2n}$, with
\[Je_1=e_2,\ \dotsc,\ Je_{2n-1}=e_{2n}, \qquad \omega = \epsilon_1 e^{12}+ \dots+ \epsilon_{n}e^{2n-1,2n}, \quad \epsilon_i=\pm1,\]
and set $D=\id$. We get
\[d\xi^*=db^*=2\omega, \qquad \ad e_0=2b^*\otimes (b+\xi)+ \sum e^i\otimes e_i.\]
Applying Corollary~\ref{cor:constructke}, one obtains the Lie algebra $\lie k$
\begin{equation}
 \label{eqn:symmetricspace}
\begin{aligned}
d e^0&=0,\\
de^i&=e^{i,0},\quad i=1,\dotsc, 2n,\\
de^{2n+1}&= 2\epsilon_1 e^{12}+ \dots+ 2\epsilon_{n}e^{2n-1,2n}+2e^{2n+1,0},
\end{aligned}
\end{equation}
with the pseudo-K\"ahler-Einstein metric
\begin{equation}
 \label{eqn:symmetricmetric}
g=\sum_{i=1}^{n}\epsilon_i (e^{2i-1}\otimes e^{2i-1}+e^{2i}\otimes e^{2i})-e^{2n+1}\otimes e^{2n+1}-e^{0}\otimes e^{0}.
\end{equation}
The resulting solvmanifold can be identified with the symmetric space
$\SU(p,q+1)/\LieG{U}(p,q)$. Indeed, fix the  diagonal matrices
\[I_{p,q}=\diag(\underbrace{1,\dotsc, 1}_{p},\underbrace{-1,\dotsc, -1}_{q}),\quad H=\diag(-1,\underbrace{1,\dotsc, 1}_{p},\underbrace{-1,\dotsc, -1}_q),\quad X=\diag(-1,\underbrace{1,\dotsc,1}_{p+q}).\]
We can identify $\su(p,q+1)$ with the Lie algebra
\[\su(p,q+1)=\left\{A\in\Sl(p+q+1,\C)\st \tran A +  H\overline{A}H=0\right\}.\]
The involution $\theta=\Ad X$
makes $(\SU(p,q+1),\LieG{U}(p,q))$ into a symmetric pair, determining a splitting $\su(p,q+1)= \lie u(p,q)\oplus \lie p$. Let $\lie a$ be the maximal abelian subalgebra of $\lie p$ spanned by
%
\[ E_0=\begin{pmatrix} 0 & 1 & 0\\
  1 & 0 & 0 \\
  0 & 0 & 0
 \end{pmatrix}.
\]
The positive eigenspaces of $\ad E_0$ generate the nilpotent Lie algebra $\lie n$ spanned by
\[E_{2n+1}=\begin{pmatrix}
	i & -i & 0 \\
	i & -i & 0 \\
	0 & 0  & 0
\end{pmatrix},\qquad
E_{2j-1}=\begin{pmatrix}
	0   & 0   & \tran{e_j}I_{p,q} \\
	0   & 0   & \tran{e_j}I_{p,q} \\
	e_j & e_j & 0
\end{pmatrix},\qquad
E_{2j}=\begin{pmatrix}
	0    & 0    & -i\tran{e_j}I_{p,q} \\
	0    & 0    & -i\tran{e_j}I_{p,q} \\
	ie_j & ie_j & 0
\end{pmatrix},
\]
where $j$ ranges between $1$ and $n$. Explicitly, we have
\[[E_0,E_{2j}]=E_{2j},\quad [E_0,E_{2j-1}]=E_{2j-1}, \quad [E_0,E_{2n+1}]=2E_{2n+1}, \quad [E_{2j-1},E_{2j}]=-2\epsilon_jE_{2n+1},\]
where $\epsilon_j$ is the $j$-th element in the diagonal of $I_{p,q}$. The semidirect product $\lie n\rtimes\lie a$ is therefore isomorphic to the Lie algebra $\lie k$ of~\eqref{eqn:symmetricspace}.
By~\cite{Tamaru_2011}, the symmetric metric can be expressed in terms of the Killing form $B$  as
\[2B_\theta|_{\lie a\times\lie a}+B_\theta|_{\lie n\times \lie n}, \qquad B_\theta(X,Y)=B(X,\theta(Y)).\]
A straightforward computation shows that this is indeed a multiple of the metric~\eqref{eqn:symmetricmetric}. For $q=0$, we obtain the positive definite symmetric metric on the Iwasawa subgroup of $\SU(n+1,\allowbreak 1)$. Suggestively, this Lie group and metric appear as the fibre of quaternion-K\"ahler manifolds obtained via the c-map (see~\cite{CortesHanMohaupt:Completeness}).

Notice that the metric~\eqref{eqn:symmetricmetric} is of pseudo-Iwasawa type; in fact, Einstein solvmanifolds arising from symmetric spaces as above are the motivating example for the notion of Iwasawa type. On the other hand, by  Proposition~\ref{prop:constructse} $\widecheck{\g}$ can also be extended to a Sasaki-Einstein Lie algebra $\widetilde{\g}$, which is not of pseudo-Iwasawa type. Explicitly,
$\widetilde{\g}$ has a basis $\{e_0,e_1,\dotsc, e_{2n+2}\}$ such that
\begin{align*}
de^0&=0,\\
de^i&=e^{i,0},\quad i=1,\dotsc, 2n,\\
de^{2n+2}&=de^{2n+1}= 2\epsilon_1 e^{12}+ \dots+ 2\epsilon_{n+1}e^{2n-1,2n}+2e^{2n+1,0},
\end{align*}
and the metric is
\[\widetilde{g}=\sum_{i=1}^{n}\epsilon_i (e^{2i-1}\otimes e^{2i-1}+e^{2i}\otimes e^{2i})-e^{2n+1}\otimes e^{2n+1}+e^{2n+2}\otimes e^{2n+2}-e^{0}\otimes e^{0}.\]
\end{example}

\begin{remark}
The pseudo-K\"ahler-Einstein quotient constructed in Example~\ref{ex:extendabelian} is precisely the family of \cite[Example~7.6]{Rossi:NewSpecialPseudoEinstein}, and since $\widecheck{\g}$ is abelian, this is consistent with Remark~\ref{rem:PKEnotPIwasawa}.
\end{remark}

\section{Classification results}
In this section we characterize $\lie z$-standard Sasaki-Einstein Lie algebras in terms of their K\"ahler reduction using the $\cu(p,q)$-Nikolayevsky derivation introduced in Section~\ref{sec:EinsteinStandard}. We also classify $\lie z$-standard Sasaki-Einstein Lie algebras of dimension $\leq 7$.

\begin{theorem}
\label{thm:metricnikSEextension}
If $\widetilde{\g}=\g\rtimes\Span{e_0}$ is a $\lie z$-standard Sasaki-Einstein Lie algebra, the $\cu(p,q)$-Nikolayevsky derivation of its K\"ahler reduction is nonzero.

Conversely, if  $\widecheck{\g}$ is a pseudo-K\"ahler Lie algebra with nonzero $\cu(p,q)$-Nikolayevsky derivation, it extends to a $\lie z$-standard Sasaki-Einstein Lie algebra $\widetilde{\g}=\g\rtimes\Span{e_0}$, uniquely determined up to equivalence.
\end{theorem}
\begin{proof}
If  $\widetilde{\g}=\g\rtimes\Span{e_0}$ is a $\lie z$-standard Sasaki-Einstein Lie algebra, Proposition~\ref{prop:constructive} asserts that $\widetilde{\g}$ can be realized as an extension of its K\"ahler reduction $\widecheck{\g}$. By Proposition~\ref{prop:constructse}, $\widecheck{D}$ is a derivation commuting with $J$ such that $\widecheck{D}^s=\id$. This implies that $\widecheck{D}$ is an element of
\[\co(2p,2q)\cap\gl(p+q,\C)=\cu(p,q)\]
with nonzero trace; if such a $\widecheck{D}$ exists, the $\cu(p,q)$-Nikolayevsky derivation is nonzero.

Now assume $\widecheck{g}$ is pseudo-K\"ahler and $\cu(p,q)$-Nikolayevsky derivation is nonzero. By rescaling, we obtain a derivation $\widecheck{D}$ whose symmetric part is the identity; this yields a Sasaki-Einstein extension by Proposition~\ref{prop:constructse}.

To prove uniqueness, fix two derivations $\widecheck{D}$, $\widecheck{D}'$ commuting with $J$, $\widecheck{D}^s=\id=(\widecheck{D}')^s$. The Lie algebras $\widecheck{\g}\rtimes_{\widecheck{D}}\Span{e_0}$ and $\widecheck{\g}\rtimes_{\widecheck{D}'}\Span{e_0}$ have a natural $\LieG{U}(p,q)$-structure. By Lemma~\ref{lemma:equivalentextensions}, they are equivalent.

We can view $\widetilde{\g}$ as an extension of $\widecheck{\g}\rtimes_{\widecheck{D}}$ by the ideal $\Span{b,\xi}$, where $\ad b=-e^0\otimes (2b+2\xi)$ and $db^*$ and $d\xi^*$ are determined by the $\LieG{U}(p,q)$-invariant form $\omega$. Therefore, $\widetilde{\g}$ and its counterpart obtained using $\widecheck{D}'$ are equivalent.
\end{proof}

In the case that $\widecheck{\g}$ is abelian, we obtain:
\begin{corollary}
\label{cor:classify_abelian}
Every $\lie z$-standard Sasaki-Einstein Lie algebra such that the K\"ahler reduction is an abelian Lie algebra is equivalent to one of those constructed in Example~\ref{ex:extendabelian}.
\end{corollary}
\begin{proof}
If $\widecheck{\g}$ is an abelian Lie algebra, we can assume $\widecheck{\g}=\R^{2n}$, with
\[Je_1=e_2,\ \dotsc,\ Je_{2n-1}=e_{2n}, \qquad \omega = \epsilon_1 e^{12}+ \dots+ \epsilon_{n}e^{2n-1,2n}, \quad \epsilon_i=\pm1;\]
the $\cu(p,q)$-Nikolayevsky derivation is $\id$, so by Theorem~\ref{thm:metricnikSEextension} the extension is equivalent to one of those constructed in Example~\ref{ex:extendabelian}.
\end{proof}

In dimension $3$, $\lie z$-standard Sasaki-Einstein Lie algebras take the form $\R^2\rtimes\Span{e_3}$, with $\ad e_3$ acting on $\R^2$ as the identity. In dimension $5$, $\lie z$-standard Sasaki-Einstein Lie algebras determine a reduction of dimension $2$, which is abelian. Therefore, these metrics have the form given in Example~\ref{ex:extendabelian}, and we obtain:
\begin{proposition}
\label{prop:5}
Let $\widetilde{\g}$ be a $\lie z$-standard Sasaki-Einstein Lie algebra of dimension $\leq 5$. Then $\widetilde{\g}$ is equivalent to one of
\begin{align*}
 &(2e^{13},2e^{13},0), & \widetilde{g}&=-e^1\otimes e^1+e^2\otimes e^2-e^3\otimes e^3,\\
&(e^{15},e^{25},2e^{12}+2e^{35},2e^{12}+2e^{35},0), & \widetilde{g}&=e^1\otimes e^1+e^2\otimes e^2-e^3\otimes e^3+e^4\otimes e^4-e^5\otimes e^5,\\
&(e^{15},e^{25},-2e^{12}+2e^{35},-2e^{12}+2e^{35},0), & \widetilde{g}&=-e^1\otimes e^1-e^2\otimes e^2-e^3\otimes e^3+e^4\otimes e^4-e^5\otimes e^5.
 \end{align*}
\end{proposition}
Note that the $5$-dimensional solvable Lie algebras appearing in Proposition~\ref{prop:5} are isomorphic; up to a sign, the metric of signature $(1,4)$ is isometric to \cite[Example~5.6]{ConDal:KillingSpinHyper}.

In dimension $7$, we can classify $\lie z$-standard Sasaki-Einstein Lie algebras by using the classification of four-dimensional Lie algebras with a  pseudo-K\"ahler metric in~\cite{Ovando:InvariantPseudoKahler}:
\begin{theorem}
\label{thm:classification}
Let $\widetilde{\g}$ be a $\lie z$-standard Sasaki-Einstein Lie algebra of dimension $7$. Then $\widetilde{\g}$ is equivalent to one of the following:
\begin{enumerate}
\item $\widetilde{\g}$ is the solvable Lie algebra
\[
(e^{17},e^{27},e^{37},e^{47},2\epsilon_1e^{12}+2\epsilon_2 e^{34}+2e^{57},2\epsilon_1e^{12}+2\epsilon_2 e^{34}+2e^{57},0)\]
with metric
\[\widetilde{g}=\epsilon_1(e^1\otimes e^1+e^2\otimes e^2)+\epsilon_2 (e^3\otimes e^3+e^4\otimes e^4)+\gamma,\quad \epsilon_1,\epsilon_2\in\{+1,-1\};\]
\item  $\widetilde{\g}$ is the solvable Lie algebra
\begin{multline*}
\Big(\frac23e^{17},\frac23e^{27},\frac a3 e^{27}+\frac43e^{37}+e^{12},-\frac a3 e^{17}+\frac43 e^{47},\\
2(e^{13}+e^{24}+ae^{12}+e^{57}),2(e^{13}+e^{24}+ae^{12}+e^{57}),0\Big)
\end{multline*}
with metric
\[\widetilde{g}= -a(e^1\otimes e^1+e^2\otimes e^2)+e^1\odot e^4-e^2\odot e^3+\gamma,\quad a\in\R;\]
\item  $\widetilde{\g}$ is the solvable Lie algebra
\begin{multline*}
\Big(\frac23e^{17},\frac23e^{27},\frac b3 e^{17}+\frac43e^{37}+e^{12},\frac b3 e^{27}+\frac43 e^{47},\\
2a(e^{13}+e^{24})+2(e^{14}-e^{23}+be^{12}+e^{57}),2a(e^{13}+e^{24})+2(e^{14}-e^{23}+be^{12}+e^{57}),0\Big)
\end{multline*}
with metric
\[\widetilde{g}=-b(e^1\otimes e^1+e^2\otimes e^2)+a(e^1\odot e^4-e^2\odot e^3)-e^1\odot e^3-e^2\odot e^4+\gamma,\quad a,b\in\R;\]
\end{enumerate}
where we have set $\gamma=- e^5\otimes e^5+e^6\otimes e^6- e^7\otimes e^7$.
\end{theorem}
\begin{proof}
By Proposition~\ref{prop:constructive}, every  $\lie z$-standard Sasaki  Lie algebra can be obtained by extending a four-dimensional pseudo-K\"ahler Lie algebra $\widecheck{\g}$. By the classification of~\cite{Ovando:InvariantPseudoKahler}, we have the following possibilities:

\textbf{1.} $\widecheck{\g}$ is abelian; we can assume that the metric is either positive-definite or neutral. Then we obtain the Lie algebras of Example~\ref{ex:extendabelian}, i.e.
\[\widetilde{\g}=(e^{17},e^{27},e^{37},e^{47},2\epsilon_1e^{12}+2\epsilon_2 e^{34}+2e^{57},2\epsilon_1e^{12}+2\epsilon_2 e^{34}+2e^{57},0)\]
with metric
\[\widetilde{g}=\epsilon_1(e^1\otimes e^1+e^2\otimes e^2)+\epsilon_2 (e^3\otimes e^3+e^4\otimes e^4)-e^5\otimes e^5+e^6\otimes e^6-e^7\otimes e^7,\]
where $\epsilon_1,\epsilon_2=\pm1$.

\textbf{2.} $\widecheck{\g}=(0,0,e^{12},0)$, with $Je_1=e_2, Je_3=e_4$, and
$\omega=e^{13}+e^{24}+ae^{12}$ for $a\in\R$.
Then
\[\widecheck{g}=-a(e^1\otimes e^1+e^2\otimes e^2)+e^1\odot e^4-e^2\odot e^3.\]
The generic $\widecheck{D}$ satisfying the hypothesis of Proposition~\ref{prop:constructse} is
\[\widecheck{D}=
\begin{pmatrix}
\frac{2}{3} & 0 & 0 & 0 \\
0 & \frac{2}{3} & 0 & 0 \\
\lambda & \frac{a}{3} & \frac{4}{3} & 0 \\
-\frac{a}{3} & \lambda & 0 & \frac{4}{3}
\end{pmatrix}.\]
By Theorem~\ref{thm:metricnikSEextension}, we can assume $\lambda=0$. Therefore, we obtain the extension
\begin{multline*}
\widetilde{\g}=\Big(\frac23e^{17},\frac23e^{27},\frac a3 e^{27}+\frac43e^{37}+e^{12},-\frac a3 e^{17}+\frac43 e^{47},\\
2e^{13}+2e^{24}+2ae^{12}+2e^{57},2e^{13}+2e^{24}+2ae^{12}+2e^{57},0\Big)
\end{multline*}
with the metric
\[\widetilde{g}= \widecheck{g} - e^5\otimes e^5+e^6\otimes e^6- e^7\otimes e^7.\]

\textbf{3.} $\widecheck{\g}=(0,0,e^{12},0)$ with $Je_1=e_2, Je_3=e_4$, and
$\omega=a(e^{13}+e^{24})+e^{14}-e^{23}+be^{12}$ for $a,b\in\R$.
Then
\[\widecheck{g}=-b(e^1\otimes e^1+e^2\otimes e^2)+a(e^1\odot e^4-e^2\odot e^3)-e^1\odot e^3-e^2\odot e^4.\]
The generic $\widecheck{D}$ satisfying the hypothesis of Proposition~\ref{prop:constructse} is
\[\widecheck{D}=
\begin{pmatrix}
 \frac{2}{3} & 0 & 0 & 0 \\
 0 & \frac{2}{3} & 0 & 0 \\
 a \lambda+\frac{b}{3} & -\lambda & \frac{4}{3} & 0 \\
 \lambda & a \lambda + \frac{b}{3} & 0 & \frac{4}{3} \\
\end{pmatrix}.\]
Again, we may assume $\lambda=0$ and obtain
\begin{multline*}
\widetilde{\g}=\Big(\frac23e^{17},\frac23e^{27},\frac b3 e^{17}+\frac43e^{37}+e^{12},\frac b3 e^{27}+\frac43 e^{47},\\
2a(e^{13}+e^{24})+2e^{14}-2e^{23}+2be^{12}+2e^{57},
2a(e^{13}+e^{24})+2e^{14}-2e^{23}+2be^{12}+2e^{57},0\Big)
\end{multline*}
with the metric
\[\widetilde{g}= \widecheck{g} - e^5\otimes e^5+e^6\otimes e^6- e^7\otimes e^7.\qedhere\]
\end{proof}

\begin{remark}
\label{rem:keexplicit}
For each of the Sasaki-Einstein Lie algebras of Theorem~\ref{thm:classification}, the center is spanned by $e_6$; taking the quotient gives explicit pseudo-K\"ahler-Einstein Lie algebras.
\end{remark}

\begin{example}
\label{ex:doesnotextend}
Consider the $6$-dimensional Lie algebra $\widecheck{\lie{g}}=(0, 0, 0, e^{12}, e^{13}, e^{14}-e^{23})$, denoted by $\lie{h}_{11}$ in~\cite{CorderoFernandezUgarte}; by \cite{Salamon:ComplexStructures,CeballosOtalUgarteVillacampa}, it admits a one-parameter family of complex structures. By the work of~\cite{CorderoFernandezUgarte}, we know that it has a four-dimensional space of compatible pseudo-K\"ahler metrics.

Instead of fixing the complex structure, we use the explicit form of the two families of pseudo-K\"ahler structures given in~\cite{Smolentsev_2013}.

The first one is $\omega_1=e^{16}-\lambda e^{25}-(\lambda-1)e^{34}$, with the compatible complex structure
\[J_1(e_2) =(1 + \lambda)ae_1,\qquad J_1(e_4) =ae_3,\qquad J_1(e_6) =\frac{(1 +\lambda)a}{\lambda} e_5\]
and metric
\[\widecheck{g}_1=-\omega_1J_1
=\frac{\lambda}{a(\lambda+1)} e^1\odot e^5+a(\lambda+1)e^2\odot e^6+\frac{1-\lambda}{a} e^3\otimes e^3+a(1-\lambda)e^4\otimes e^4,\]
while the second one is $\omega_2=e^{16}+e^{24}-\frac12(e^{25}-e^{34})$ with the compatible complex structure
\[ J_2(e_2) =-ae_1,\qquad J_2(e_3)=\frac{3}{2a}e_4+\frac3ae_5, \qquad J_2(e_4) =-\frac23a e_3-\frac1ae_6, \qquad J_2(e_6) =-2a e_5
\]
and metric $\widecheck{g}_2=-\omega_2J_2$.

The first case, imposing $[\widecheck{D}, J_1]=0$ gives
\[\widecheck{D}=
\begin{pmatrix}
  \frac{\mu_1}{3}&0&0&0&0&0\\
  0&\frac{\mu_1}{3}&0&0&0&0\\
  \frac{\mu_2}{b}&-a^2\frac{\mu_3}{b}(b+1)&\frac{2\mu_1}{3}&0&0&0\\
  \frac{\mu_3}{b}&\mu_2+\frac{\mu_2}{b}&0&\frac{2\mu_1}{3}&0&0\\
  \frac{\mu_4}{b}&-a^2\frac{\mu_5}{b}(b+1)^2&\mu_2+\frac{\mu_2}{b}&-a^2\frac{\mu_3}{b}(b+1)&\mu_1&0\\
  \mu_5&\mu_4&\mu_3&\mu_2&0&\mu_1
\end{pmatrix}
\]
and imposing $\widecheck{D}^s=\id$ gives $\mu_1=\frac32$ and $\mu_i=0$ for $i=2,\dots,5$, that is
\[\widecheck{D}=
\begin{pmatrix}
  \frac{1}{2}&0&0&0&0&0\\
  0&\frac{1}{2}&0&0&0&0\\
  0&0&1&0&0&0\\
  0&0&0&1&0&0\\
  0&0&0&0&\frac{3}{2}&0\\
  0&0&0&0&0&\frac{3}{2}
\end{pmatrix}.
\]
Writing $e_7,e_8,e_9$ instead of $b,\xi,e_0$, the metric on $\widetilde{\g}$ is then $\widetilde{g}=\widecheck{g}_1-e^7\otimes e^7+e^8\otimes e^8-e^9\otimes e^9$, whilst the Lie algebra is
\begin{multline*}
\widetilde{\g}=(\frac{1}{2}e^{19},\frac12e^{29},e^{39},e^{12}+e^{49},e^{13}+\frac{3}{2}e^{59},e^{14}-e^{23}+\frac32e^{69},\\
2e^{16}-2\lambda e^{25}-2(\lambda-1)e^{34}+2e^{79},2e^{16}-2\lambda e^{25}-2(\lambda-1)e^{34}+2e^{79},0)
\end{multline*}
On the other hand, $[\widecheck{D}, J_2]=0$ gives
\[\widecheck{D}=
\begin{pmatrix}
  \frac{\mu_1}{3}&0&0&0&0&0\\
  0&\frac{\mu_1}{3}&0&0&0&0\\
  2\mu_2&-\frac{2}{3}(2a^2\mu_3+\mu_1)&\frac{2\mu_1}{3}&0&0&0\\
  2\mu_3+\frac{\mu_1}{a^2}&3\mu_2&0&\frac{2\mu_1}{3}&0&0\\
  2(\mu_4+2\mu_3+\frac{\mu_1}{a^2})&-2(a^2\mu_5+3\mu_2)&3\mu_2&-\frac23(2a^2\mu_3+\mu_1)&\mu_1&0\\
  \mu_5&\mu_4&\mu_3&\mu_2&0&\mu_1
\end{pmatrix}
\]
but imposing $\widecheck{D}^s=\id$ does not yield any solution for the $\mu_i$.
\end{example}

\begin{example}
The following example shows a $\lie z$-standard Sasaki-Einstein $\widetilde{\g}$ obtained by extending a $6$-dimensional pseudo-K\"ahler Lie algebra with a derivation $\widecheck{D}$ which is not a multiple of the $\cu(p,q)$-Nikolayevsky derivation. Consider the Lie algebra $\widecheck{\g}=(0, 0, e^{12}, 0, 0, 0)$ with symplectic form $\omega=e^{13}+e^{24}+e^{56}$ and complex structure $J(e_1)=e_2$, $J(e_3)=e_4$ and $J(e_5)=e_6$. Then
\[\widecheck{D}=
\begin{pmatrix}
\frac23&0&0&0&0&0\\
0&\frac23&0&0&0&0\\
\mu&0&\frac43&0&\lambda&-\nu\\
0&\mu&0&\frac43&\nu&\lambda\\
\nu&-\lambda&0&0&1&-\rho\\
\lambda&\nu&0&0&\rho&1
\end{pmatrix}
\]
satisfies the hypothesis of Proposition~\ref{prop:constructse}, and therefore determines a $\lie z$-standard Sasaki-Einstein $\widetilde{\g}$ Lie algebra of dimension $9$. The derivation $\widecheck{D}$ is not diagonalizable over $\R$, but has eigenvalues $(\frac23,\frac23,1-i\rho,1+i\rho,\frac43,\frac43)$; therefore, $\widecheck{D}$ is only a multiple of the $\cu(p,q)$-Nikolayevsky derivation when $\rho$ is zero. Note, however, that all the resulting extensions are isometric by Theorem~\ref{thm:metricnikSEextension}.
\end{example}

\bibliographystyle{plainurl}
\bibliography{PKPSEinsteinSolv_ContiRossiSegnan}

\medskip
\small\noindent D. Conti: Dipartimento di Matematica, Università di Pisa, largo Bruno Pontecorvo 6, 56127 Pisa, Italy.\\
\texttt{diego.conti@unipi.it}\\
\small\noindent R.~Segnan Dalmasso: Dipartimento di Matematica e Applicazioni, Universit\`a di Milano Bicocca, via Cozzi 55, 20125 Milano, Italy.\\
\texttt{r.segnandalmasso@campus.unimib.it}\\
\small\noindent F. A. Rossi: Dipartimento di Matematica e Informatica, Universit\`a degli studi di Perugia, via Vanvitelli 1, 06123 Perugia, Italy.\\
\texttt{federicoalberto.rossi@unipg.it}
\end{document}